\documentclass[11pt,a4paper, leqno]{amsart}

\usepackage[T1]{fontenc}
\usepackage{amsfonts, amssymb, amsthm}
\usepackage{url}      
\usepackage[numbers]{natbib}   
\usepackage{mathrsfs}

\usepackage{tikz}
\usetikzlibrary{matrix,arrows,decorations.pathmorphing}

\newtheorem{theorem}{Theorem}[section]
\newtheorem{lemma}[theorem]{Lemma}
\newtheorem{corollary}[theorem]{Corollary}
\newtheorem{proposition}[theorem]{Proposition}
\theoremstyle{definition}
\newtheorem{remark}[theorem]{Remark}

\newcommand{\cA}{\mathcal{A}}
\newcommand{\cE}{\mathcal{E}}
\newcommand{\cN}{\mathcal{N}}
\newcommand{\cM}{\mathcal{M}}
\newcommand{\lin}[1]{\mathcal{L}(#1)}
\newcommand{\csp}[1]{\mathcal{C}_c(#1)}
\newcommand{\CrKS}[1]{C^*_{r,\mathrm{KS}}(#1)}
\newcommand{\CKS}[1]{C^*_{\mathrm{KS}}(#1)}
\newcommand{\CrAlt}[1]{C^*_{r,\mathrm{alt}}(#1)}

\newcommand{\fEKS}{\mathfrak{E}_{\rm KS}}

\newcommand{\fEr}{\mathfrak{E}_{\rm r}}
\newcommand{\fE}{\mathfrak{E}}

\begin{document}

\title[Inverse semigroup Fell bundles and regular representations]{On Fell bundles over inverse semigroups and their left regular representations}
\author{Erik B{\'e}dos}
\address{Erik B{\'e}dos, Institute of Mathematics, University of Oslo, P.b.~1053 Blindern, 0316 Oslo, Norway.}
\email{bedos@math.uio.no}
\author{Magnus D. Norling}
\address{Magnus D. Norling, Oslo, 
Norway.}
\email{magnus.dahler.norling@gmail.com}
\subjclass[2010]{Primary 46L55; Secondary 46L05, 20M18}
\keywords{Fell bundles, inverse semigroups, regular representation, cross sectional $C^*$-algebras}
\begin{abstract}
We prove a version of Wordingham's theorem for left regular representations in the setting of Fell bundles of inverse semigroups and use this result to discuss the various associated cross sectional $C^*$-algebras.  
\end{abstract}
\date{\today}

\maketitle
\section{Introduction}

\medskip Following an unpublished work of Sieben, the concept of  a Fell bundle over a discrete group was generalized by Exel in \cite{exel11}, where the notion of a Fell bundle $\cA = \{A_s\}_{s\in S}$ over an inverse semigroup $S$ was introduced and the associated full cross sectional $C^*$-algebra $C^*(\cA)$ was defined as the universal $C^*$-algebra for $C^*$-algebraic representations of $\cA$. 
This construction may be used to present other classes of  $C^*$-algebras in a unified manner. For example, Buss and Exel  show in \cite{buss_exel11} that to each partial action $\beta$ of an inverse semigroup $S$ on a C$^*$-algebra $A$,  one may associate a Fell bundle $\cA_\beta$ over $S$ such that   
$C^*(\cA_\beta)$ is naturally isomorphic to the full crossed product of $A$ by $\beta$. (In fact, Buss and Exel even consider twisted partial actions.)
 In another direction,  the same authors establish in 
 \cite{buss_exel_12a} that any Fell bundle $\mathcal{B}$ over an {\'e}tale groupoid $\mathcal{G}$ gives rise to a Fell bundle $\mathcal{A}$ over $S$, where $S$ is any inverse semigroup consisting of bisections (or slices) of $\mathcal{G}$ (as defined by Renault in \cite{re80}),  and show that, under some mild assumptions, the full cross sectional C$^*$-algebra of $\mathcal{B}$ is isomorphic to $C^*(\cA)$. 
 
 Exel also defines in \cite{exel11} the reduced cross sectional $C^*$-algebra  $C^*_r(\cA)$ associated to a Fell bundle $\cA$ over an inverse semigroup $S$. One drawback of his construction is that it is somewhat involved (we summarize it in section \ref{sec:exelreduced}). Another approach has recently been proposed in   \cite{buss_exel_meyer15} when $\cA$ is saturated and $S$ is unital, relying on the result from \cite{buss_meyer14} that $\cA$ may then be identified with an action of $S$ by Hilbert bimodules on the unit fibre $A$ of $\cA$, thus making it possible to define $C^*_r(\cA)$ as the reduced crossed product of $A$ by this action of $S$.  We believe that it would be helpful to find a more direct construction of $C^*_r(\cA)$, at least in some cases. For example, such a construction would make it easier to initiate a study of amenability for Fell bundles over inverse semigroups, having in mind that, 
 as for Fell bundles over groups \cite{exel14, exel97}, a natural definition of  amenability for $\cA$   is to require that the canonical $*$-homomorphism from $C^*(\cA)$ onto $C_r^*(\cA)$ is injective. 

 Given a Fell bundle $\cA$ over an inverse semigroup $S$, our main goal in the present paper is to introduce  a  certain $C^*$-algebra  $C^*_{\rm{r, alt}}(\cA)$  which may also be considered as
 a kind of reduced cross sectional C$^*$-algebra for $\cA$,
 and to compare it with Exel's $C_r^*(\cA)$.  
 Inspired by the approach used by 
 Khoshkam and Skandalis \cite{khoshkam_skandalis04} in the case of an action of an inverse semigroup on a $C^*$-algebra, 
 our first step (see Section \ref{Fellb}) is to associate  
 to $\cA$ a full cross sectional $C^*$-algebra $C^*_{\rm KS}(\cA)$ which is universal for so-called pre-representations of $\cA$ in $C^*$-algebras, or, equivalently, for $C^*$-algebraic representations of the convolution $*$-algebra $C_c(\cA)$ canonically attached to $\cA$. Exel's $C^*(\cA)$ is then easily obtained as a quotient of $C^*_{\rm KS}(\cA)$. 
  Our next step (see Section \ref{leftregrep}) is to show that $C_c(\cA)$ has a natural injective left regular $C^*$-algebraic representation $\Phi_\Lambda$. 
 The injectivity of $\Phi_\Lambda$ may be seen as an analog of Wordingham's theorem for $\ell^1(S)$ (cf.~\cite{paterson}), and our proof is related to his original proof, although some extra arguments are necessary.  Letting $C_{\rm r, KS}(\cA)$ denote the $C^*$-algebra generated by the range of $\Phi_\Lambda$, $C^*_{\rm{r, alt}}(\cA)$ is then defined as the quotient of $C_{\rm r, KS}(\cA)$ by a certain canonical ideal. From the naturality of our construction, it readily follows that there is a canonical $*$-homomorphism $\Psi_{\Lambda^{\rm alt}}$ from $C^*(\cA)$ onto $C^*_{\rm{r, alt}}(\cA)$.

 In the case where $S$ consists only of idempotents (hence is a semilattice) and $\cE$ is Fell bundle over $S$, we check in Section \ref{semilatt} that $C^*_{\rm KS}(\cE) = C^*_{\rm r, KS}(\cE)$ and $C^*(\cE) = C^*_{\rm{r, alt}}(\cE) \simeq C^*_{\rm{r}}(\cE)$. Next, given a Fell bundle $\cA$ over an inverse semigroup $S$ such that $S$ is $E^*$-unitary (cf.~Section \ref{prem}) and such that  $A_0 =\{0\}$ if $S$ has a 0 element, we let $\cE$ denote the Fell bundle obtained by restricting $\cA$ to the semilattice $E$ of idempotents in $S$ and  show (in Section \ref{cond-exp-diag}) that there exists a faithful conditional expectation from  $C^*_{\rm r, KS}(\cA)$  onto $C^*_{\rm r, KS}(\cE)$. In the final section, keeping the same assumptions, we describe $C_r^*(\cA)$ as a quotient of $C_{\rm r, KS}(\cA)$ and show that there exists a surjective canonical $*$-homomorphism $\Psi'$ from $C^*_{\rm{r, alt}}(\cA)$ onto $C_r^*(\cA)$. We also   
characterize when $\Psi'$ is a $*$-isomorphism and end by showing that this happens frequently when $S$ is strongly $E^*$-unitary. 
  
 \section{Preliminaries}\label{prem}
 We recall that a semigroup is a set equipped with an associative binary operation, while  a monoid is a semigroup with an identity. A commutative idempotent semigroup is called a semilattice.  
 We also recall that an inverse semigroup is a semigroup $S$ where for each $s\in S$ there is a unique $s^* \in S$ satisfying
\[
 ss^*s= s\mbox{ and }s^*ss^* = s^*.
\]
The map $s\mapsto s^*$ is then an involution on $S$. Every inverse semigroup $S$ contains a canonical semilattice, namely
\[E(S)=\{e\in S:e^2=e\}\] satisfying $E(S) = \{e\in S:e^2=e = e^*\} = \{s^*s:s\in S\}=\{tt^*:t\in S\}.$ Throughout this article,  $S$ will be a fixed inverse semigroup and  $E=E(S)$ will denote its semilattice of idempotents. We refer to \cite{lawson} and \cite{paterson} for the basics of the theory of inverse semigroups. We recall below a few facts that we will need later.

There is a natural partial order relation $\leq$ on $S$ given by 
$s\leq t$ if and only if $s=et$ for some $e\in E$, if and only if $s=tf$ for some $f\in E$, where $e$ may be chosen to be $ss^*$, and $f$ to be $s^*s$.  For $e,f \in E$, we have $e \leq f$ if and only if $e = ef$. 

Many inverse semigroups have a zero, that is, an element $0$  satisfying $0s = s0 = 0$ for all $s\in S$. Such an element is  necessarily unique and lies in $E$. If $S$ has a zero, we set $S^\times = S\setminus \{0\}$ and $E^\times = E\setminus \{0\}$. Otherwise, we set $S^\times = S$ and $E^\times = E$.

We will say that $S$ is \emph{$E^*$-unitary} if the set $\{ s\in S: e\leq s\}$ is contained in $E$ for every $e \in E^\times$. If this holds for every $e \in E$, then $S$ is called \emph{$E$-unitary}. These two concepts clearly coincide if $S$ does not have a zero. For inverse semigroups having a zero, $E$-unitarity is a too strong requirement, only satisfied by semilattices. We note that $E^*$-unitarity is usually only defined for inverse semigroups having a zero, in which case it is sometimes called $0$-$E$-unitarity (and is defined as above). Our use of terminology will allow us to unify some statements.  The class of $E^*$-unitary inverse semigroups has by far been the one who has received most attention from $C^*$-algebraists, most probably because they are easier to handle.

We will also need to refer to a stronger form of $E^*$-unitarity. 
We recall that a map $\sigma$ from $S$ into a group with identity $1$ is called a \emph{grading}  if $\sigma(st)=\sigma(s)\sigma(t)$ whenever $s, t \in S$ and $st\in S^\times$, and that $\sigma$ is said to be \emph{idempotent pure} if $\sigma^{-1}(\{1\})=E$. Then $S$ is said to be \emph{strongly $E^*$-unitary} \cite{fleming_fountain_gould99} 
if there exists an idempotent pure grading from $S$ into some group. It is known that  $S$ is  $E^*$-unitary whenever it is strongly $E^*$-unitary. 
   
A semigroup homomorphism from $S$ into another inverse semigroup is necessarily  $*$-preserving, so this provides the natural notion of homomorphism between inverse semigroups. 
If $X$ is a set, then $\mathcal{I}(X)$ will denote the symmetric inverse semigroup on $X$, consisting of all partial bijections on $X$ (with composition defined on the largest possible domain). An \emph{action of $S$ on $X$} is then a homomorphism of $S$ into $\mathcal{I}(X)$. The essence of the Wagner-Preston theorem is that there always exists an injective action of $S$ on some set.

 When $A$ is a $C^*$-algebra, we will let ${\rm PAut}(A)$ denote the inverse subsemigroup of $\mathcal{I}(A)$ consisting of all partial $*$-automorphisms of $A$; so $\phi \in {\rm PAut}(A)$ if and only if $\phi$ is a $*$-isomorphism between two  ideals of $A$. Here, and in the sequel, ideals in $C^*$-algebras are always assumed to be two-sided and closed, unless otherwise specified. An \emph{action of $S$ on a $C^*$-algebra $A$} is an homomorphism $\alpha$ from $S$ into ${\rm PAut}(A)$. 
 Khoshkam and Skandalis show in \cite{khoshkam_skandalis04} how to associate to such an action a full (resp.~reduced) $C^*$-crossed product, which we will denote by $A\rtimes^{\rm KS}_{\alpha} S$ (resp.~$A\rtimes^{\rm KS}_{\alpha, r} S$).  In fact, their construction goes through for a more general kind of action of $S$ on $A$, and the interested reader should consult \cite{khoshkam_skandalis04} for more details, including a discussion of the relationship between their full crossed product and the crossed product construction previously introduced by Sieben in \cite{sieben}.
 
Following \cite{buss_exel12}, one may define \emph{partial} actions of $S$ on sets and on $C^*$-algebras. As mentioned in the introduction, Buss and Exel actually consider  \emph{twisted} partial actions of $S$ in \cite{buss_exel12}, but we will restrict ourselves to the untwisted case to avoid many technicalities.  Partial actions were first introduced in the case where $S$ is a group, and the reader may consult \cite{exel14} for a nice introduction to this subject, including many references to the literature. We recall a few relevant definitions and facts from \cite{buss_exel12}. 
 A \emph{partial homomorphism} of $S$ in a semigroup $H$ is a map $\pi: S\mapsto H$ such that 
\begin{enumerate}
\item $\pi(s)\pi(t)\pi(t^*) = \pi(st)\pi(t^*)$,
\item  $\pi(s^*)\pi(s)\pi(t) = \pi(s^*)\pi(st)$,
\item $\pi(s)\pi(s^*)\pi(s) = \pi(s)$
\end{enumerate}
hold for all $s,t\in S$. Note that if $H$ is an inverse semigroup, then (iii) implies 
\begin{enumerate}\setcounter{enumi}{3}
\item $\pi(s^*)=\pi(s)^*$
\end{enumerate}
for all $s\in S$; hence, in this case, $\pi$ is a partial homomorphism if and only if (i), (ii) and (iv) hold. Moreover, still assuming that $H$ is an inverse semigroup,  this is equivalent to requiring that the three conditions 
\begin{itemize}
\item[(a)] $\pi(s^*)=\pi(s)^*$,
\item[(b)] $\pi(s)\pi(t)\leq \pi(st)$,
\item[(c)] $\pi(s)\leq \pi(t)$ \ whenever $s\leq t$,
\end{itemize}
hold for $s,t \in S$, cf.\ \cite[Proposition 3.1]{buss_exel12}. 

A \emph{partial action} of $S$ on a set $X$ (resp.~on a $C^*$-algebra $A$) is then defined as a partial homomorphism $\beta$ from $S$ into $\mathcal{I}(X)$ (resp.~ into ${\rm PAut}(A)$).  As in \cite{buss_exel12}, we will also require that a partial action $\beta$ of $S$ on a $C^*$-algebra $A$ satisfies that the union $\cup_{s\in S} J_s = \cup_{s\in S}\, {\rm im}({\beta_s})$ spans a dense subspace of $A$.

\section{Fell bundles over inverse semigroups} \label{Fellb}
 In \cite{exel11}, Exel defines a Fell bundle over $S$ as a quadruple
\[
 \cA = \big(\{A_s\}_{s\in S} , \{\mu_{s,t}\}_{s,t\in S} , \{*_s\}_{s\in S} , \{j_{t,s}\}_{s,t\in S, s\leq t}\big)
\]
where for $s,t\in S$ we have that
\begin{enumerate}
 \item $A_s$ is a complex Banach space,
 \item $\mu_{s,t}:A_s\odot A_t\to A_{st}$ is a linear map,
 \item $*_s:A_s\to A_{s^*}$ is a conjugate linear isometric map,
 \item $j_{t,s}:A_s\hookrightarrow A_t$ is a linear isometric 
 map whenever $s\leq t$.
\end{enumerate}
It is moreover required that for every $r,s,t\in S$, and every $a\in A_r,b\in A_s$, and $c\in A_t$, we have
\begin{enumerate}\setcounter{enumi}{4}
 \item $\mu_{rs,t}(\mu_{r,s}(a\otimes b)\otimes c)=\mu_{r,st}(a\otimes\mu_{s,t}(b\otimes c))$,
 \item $*_{rs}(\mu_{r,s}(a\otimes b))=\mu_{s^*,r^*}(*_s(b)\otimes *_r(a))$,
 \item $*_{s^*}(*_s(a))=a$,
 \item $\|\mu_{r,s}(a\otimes b)\|\leq\|a\| \, \|b\|$,
 \item $\|\mu_{r^*,r}(*_r(a)\otimes a)\| = \|a\|^2$,
 \item $\mu_{r^*,r}(*_r(a)\otimes a) \geq 0$ in $A_{r^* r}$,
 \item if $r\leq s\leq t$, then $j_{t,r}=j_{t,s}\circ j_{s,r}$,
 \item  if $r\leq r'$ and $s\leq s'$, then $j_{r's',rs}\circ\mu_{r,s}=\mu_{r',s'}\circ (j_{r',r}\otimes j_{s',s})$ and $j_{s',s}\circ *_s = *_{s'}\circ j_{s',s}$.
\end{enumerate}
As shown by Exel, axioms (i)-(ix) imply that $A_e$ is a $C^*$-algebra whenever $e \in E$, with $cd= \mu_{e,e}(c \otimes d)$ and $c^*= *_e(c)$ for $c,d \in A_e$.
 Hence, the requirement in axiom (x) is meaningful. Exel also shows that the following properties hold:
\begin{enumerate}
 \item[(xiii)] $j_{s,s}$ is the identity map ${\rm id}_{A_s}$ for every $s\in S$;
 \item[(xiv)] If $e,f\in E$ and $e\leq f$, then $j_{f,e}(A_e)$ is an 
 ideal in $A_f$.
\end{enumerate}

When no confusion is possible, we will use the simplified notation
\[
 a\cdot b:=\mu_{s,t}(a\otimes b)\mbox{ and }a^*:=*_s(a)
\]
whenever $a\in A_s$ and $b\in A_t$, and just write  $\cA = (\{A_s\}_{s\in S} , \{j_{t,s}\}_{s,t\in S,\, s\leq t})$, or  even only $\cA = \{A_s\}_{s\in S}$ in some cases. If $s \in S$ and $e:= s^*s \in E$, then one easily verifies that $A_s$ becomes a right Hilbert $A_e$-module with respect to the right action given by $(a_s, a_e)\mapsto a_s \cdot a_e \in A_{ss^*s} = A_s$ for $a_s \in A_s$ and $a_e \in A_e$,  and the $A_e$-valued inner product given by $\langle a_s, b_s \rangle = a_s^* \cdot b_s$ for $a_s, b_s \in A_s$.
For later use, we also note the following fact: 
\begin{equation}\label{a-e-f} 
\text{If } e, f\in E, e\leq f, \text{and }c,d \in A_e, \text{then } j_{f,e}(c)\cdot d = cd\,.
\end{equation} 
Indeed, using axioms (xii) and (xiii), we get 
\[j_{f,e}(c)\cdot d = j_{f,e}(c)\cdot j_{e,e}(d) = j_{fe,ee}(c\cdot d) = j_{e,e}(c \cdot d) = c\cdot d = cd\,.\]

We also recall that a Fell bundle $\cA$ is called \emph{saturated} when the span of $\{a_s\cdot b_t : a_s\in A_s, b_t\in A_t\}$ is dense in $A_{st}$ for all $s,t \in S$.

\subsection{}\label{FB_PA}
An important class of examples of Fell bundles over inverse semigroups arises from (twisted) partial actions of inverse semigroups on $C^*$-algebras (cf.~\cite[Section 8]{buss_exel12}). For the ease of the reader, we sketch this construction in the untwisted case.

Let $\beta:S\to {\rm PAut}(A)$ be a partial action of $S$ on a $C^*$-algebra $A$. For each $s\in S$, set $J_{s^*}={\rm dom}(\beta_s)$ and $J_s={\rm im}(\beta_s)$, so $\beta_s$ is a $*$-isomorphism from $J_{s^*}$ onto $J_s$ and $\beta_s^{-1} = \beta_{s^*}$. One may then show (cf.~\cite[Proposition 3.4, Proposition 3.8 and Proposition 6.3]{buss_exel12}) that the family $\{J_s\}$ satisfy certain compatibility properties, such as $\beta_s(J_{s^*} \cap J_t) = J_s \cap J_{st}\,$,\, $J_s \subset J_{ss^*}$ (and $J_s = J_{ss^*}$ if $\beta$ is an action),\, $J_s \subset J_t$ whenever $s\leq t$, and $\beta_e = {\rm id}_{J_e}$ when $e\in E$.

Now, for each $s\in S$, we set $A_s= \{(a, s): a\in J_s\}$ and organize $A_s$ as a Banach space by identifying $a\in J_s$ with $(a,s)\in A_s$. Note that if $s,t \in S$, $a \in J_s$ and $b\in J_t$, then 
\[\beta_s^{-1}(a)\,b \in J_{s^*} \cap J_t\,, \,\text{so} \, \,\beta_s\Big(\beta_s^{-1}(a)b\Big)  \in J_s \cap J_{st}\,,\, \]
\[\text{and}\,\, a^* \in J_s\,,\, \text{so} \, \,\beta_s^{-1}(a^*) \in J_{s^*}\,;\]  
hence one may define $\mu_{s,t}\big((a,s)\otimes (b,t)\big)=:(a,s) \cdot (b,t)$ and $*_s(a,s)=: (a,s)^*$ by 
\[ (a,s) \cdot (b,t) = \Big(\beta_s\big(\beta_s^{-1}(a)b\big), st\Big) \in A_{st} \, \, \text{and} \,\, (a,s)^* = \big(\beta_s^{-1}(a^*), s^*) \in A_{s^*}\,.\]
Moreover, if $s\leq t$\,, then $\beta_s \leq \beta_t$\,, so $J_s = {\rm im}(\beta_s) \subset {\rm im}(\beta_t) = J_t $, and one may then define $j_{t,s}: A_s\to A_t$ by 
\[j_{t,s} (a,s) = (a, t)\,, \quad \text{for all}\,\,  a \in J_s\,.\]
It may then be checked that $\cA = (\{A_s\}_{s\in S}, \{j_{t,s}\}_{s,t\in S,\, s\leq t})$ becomes a Fell bundle over $S$ with respect to these operations (cf.\ \cite{buss_exel11} for the case of a global (twisted) action).

\subsection{} \label{exelfull} Still following \cite{exel11}, a \emph{pre-representation} of a Fell bundle \[\cA=(\{A_s\}_{s\in S} , \{j_{t,s}\}_{s,t\in S,\, s\leq t})\] in a complex $*$-algebra $B$ is a family $\Pi=\{\pi_s\}_{s\in S}$, where for each $s\in S$,
\[
 \pi_s:A_s\to B
\]
is a linear map such that for all $s,t\in S$, all $a\in A_s$, and all $b\in A_t$, one has
\begin{enumerate}
 \item[(a)] $\pi_{st}(a\cdot b)=\pi_s(a)\pi_t(b)$, 
 
 \smallskip
 \item[(b)] $\pi_{s^*}(a^*)=\pi_s(a)^*$.
\end{enumerate}
If, in addition,  $\Pi$ satisfies
\begin{enumerate}
 \item[(c)] $\pi_t\circ j_{t,s}=\pi_s$ whenever $s\leq t$,
\end{enumerate}
then  $\Pi$ is called a {\it representation} of $\cA$ in $B$. 

We recall that if $\Pi$ is a pre-representation of $\cA$ in a $C^*$-algebra $B$, then for each $s\in S$ and $a\in A_s$, we have $\|\pi_s(a)\| \leq \|a\|$. Indeed, as $\pi_e:A_e\mapsto B$ is then a $*$-homomorphism between $C^*$-algebras for every $e \in E$,  we get
\[\|\pi_s(a)\|^2 = \|\pi_s(a)^*\pi_s(a)\| = \|\pi_{s^*}(a^*)\pi_s(a)\| = \|\pi_{s^*s}(a^*\cdot \,a)\| \leq \|a^*\cdot \,a\| =\|a\|^2.\]  

\medskip Consider now the direct sum of vector spaces 
\[
 \csp{\cA} = \bigoplus_{s\in S} A_s
\]
We will often write an element $g \in   \csp{\cA}$
as a formal sum
$
 g=\sum_{s\in S}a_s\delta_s
$
where $a_s \in A_s$ for $s\in S$ and all but finitely many $a_s$ are equal to $0$. Then $\csp{\cA}$ can be given the structure of a complex $*$-algebra by extending linearly the operations
\begin{align*}
 (a_s\delta_s)(b_t\delta_t) &= (a_s\cdot b_t)\delta_{st}\\ 
 (a_s\delta_s)^* &= a_s^*\delta_{s^*}
\end{align*}
Alternatively, if one prefers to write $\csp{\cA}$ as 
\[\csp{\cA}=\Big\{ g\in \prod_{s\in S} A_s :  g(s) = 0 \,\,\text{for all but finitely many}\, s\Big\},\]
one may define the product and the involution on $\csp{\cA}$ by
\[(f\ast g)(r) = \sum_{s,\, t \in S,\, st \,=\, r } \,f(s)\cdot g(t)\quad \text{and} \quad f^*(r) = f(r^*)^*\] 
for $f, g \in \csp{\cA}$ and $r\in S$.

For each $s\in S$, let $\pi^0_s:A_s \to \csp{\cA}$ be defined by $$\pi^0_s(a_s) = a_s \delta_s$$ for each $a_s \in A_s$. Then $\Pi^0:=\big\{\pi^0_s\big\}_{s\in S}$ is a pre-representation of $\cA$ in $\csp{\cA}$, which satisfies the following universal property (cf.~\cite[Proposition 3.7]{exel11}):

 To each pre-representation $\Pi=\{\pi_s\}_{s\in S}$ of $\cA$ in a $*$-algebra $B$ one may associate a $*$-homomorphism $\Phi_\Pi:\csp{\cA}\to B$ given by
\[
 \Phi_\Pi\Big(\sum_{s\in S}a_s\delta_s\Big)=\sum_{s\in S}\pi_s(a_s)\,,
\]
which satisfies  $\Phi_\Pi\,\circ\, \pi_s^0 = \pi_s$ for all $s\in S$. Moreover, the map $\Pi \mapsto \Phi_\Pi$ gives a bijection between pre-representations of $\cA$ in $B$ and $*$-homomophisms from $\csp{\cA}$ into $B$. 

Consider $g=\sum_{s\in S}a_s\delta_s \in \csp{\cA}$. If $B$ is a $C^*$-algebra and $\Pi$ is a pre-representation of $\cA$ in $B$, then we have
\[\|\Phi_\Pi(g)\| = \big\|\sum_{s\in S}\pi_s(a_s)\big\| \leq \sum_{s\in S}\|\pi_s(a_s)\| \leq \sum_{s\in S}\,\|a_s\|\,.\]
Hence, if we define
 \[\|g\|_{\rm u} := \sup_\Phi \big\{ \|\Phi(g)\| \}\,,\] 
 where the supremum is taken over all  $*$-homomorphisms  from $\csp{\cA}$ into any $C^*$-algebra, then 
\begin{align*}
\|g\|_{\rm u} &= \sup\big\{ \|\Phi_\Pi(g)\| : 
\Pi \ \text{is a pre-representation of $\cA$ in some $C^*$-algebra}\big\} \\
&\leq \, \sum_{s\in S}\,\|a_s\| \,< \,\infty\,.
\end{align*}
Hence $\|\cdot\|_{\rm u}$ gives a $C^*$-seminorm on $\csp{\cA}$. As we will show in the next section, there always exists an injective $*$-representation of $\csp{\cA}$ 
in some $C^*$-algebra
(namely the one associated to the left regular representation of $\csp{\cA}$). 
It follows that   $\|\cdot\|_{\rm u}$ is in fact a $C^*$-norm, and we  
may therefore define the \emph{full KS-cross sectional $C^*$-algebra of $\cA$}, denoted by $C^{*}_{\rm KS}(\cA)$,  as the completion of  $\csp{\cA}$ w.r.t.~$\|\cdot\|_{\rm u}$. 

We will use the same notation to denote the norm on $C^{*}_{\rm KS}(\cA)$ and will identify $\csp{\cA}$ with its canonical copy in $C^{*}_{\rm KS}(\cA)$. We may therefore regard $\Pi^0$ as a pre-representation of $\cA$ in $C^{*}_{\rm KS}(\cA)$, which is universal in the sense that
given any  pre-representation $\Pi$ of $\cA$ in a $C^*$-algebra $B$, then
 there exists a unique $*$-homomorphism from $C^{*}_{\rm KS}(\cA)$ into $B$, which we also denote by $\Phi_{\Pi}$, satisfying $\Phi_\Pi\, \circ\, \pi^0_s= \pi_s$ for all $s\in S$.

If for example $\cA_\alpha$ denotes the Fell bundle associated to an action $\alpha$ of $S$ on a $C^*$-algebra $A$, then it is straightforward to verify that $C^{*}_{\rm KS}\big(\cA_\alpha\big)$ coincides with the full KS-crossed product $A\rtimes^{\rm KS}_\alpha S$ constructed in \cite{khoshkam_skandalis04}. 
 Thus, if  $\beta$ is a partial action of $S$ on a $C^*$-algebra $A$,  it is natural to define     
 the \emph{full KS-crossed product}  
 by $A\rtimes^{\rm KS}_\beta S:=C^{*}_{\rm KS}\big(\cA_\beta\big),$ where $\cA_\beta$ denotes the Fell bundle over $S$ associated to $\beta$ in \ref{FB_PA}.

\subsection{}\label{sec:fullcross} In \cite{exel11}, Exel defines the full cross sectional $C^*$-algebra $C^{*}(\cA)$  of  a Fell bundle 
$\cA = \big(\{A_s\}_{s\in S}, \{j_{t,s}\}_{s,t\in S,\, s\leq t}\big)$.
 This algebra may be described as a quotient of $C^{*}_{\rm KS}(\cA)$. To explain this, we first have to review Exel's construction. Let $\cN_\cA$ denote the subspace of $\csp{\cA}$ spanned by the set
\[
 \Big\{a_s\delta_s-j_{t,s}(a_s)\delta_t:s,t\in S, s\leq t, a_s\in A_s\Big\}.
\]
Exel shows in \cite[Proposition 3.9]{exel11} that $\cN_\cA$ is a two-sided selfadjoint ideal of $\csp{\cA}$. It follows that $\csp{\cA}/\cN_\cA$ becomes a complex $*$-algebra in the obvious way. Moreover, \cite[Proposition 3.10]{exel11} says that if $\Pi$ is a pre-representation of $\cA$ in a $*$-algebra $B$, then $\Pi$ is a representation of $\cA$ if and only if $\Phi_\Pi$ vanishes on $\cN_\cA$, in which case 
we will denote the associated $*$-homomorphism from $\csp{\cA}/\cN_\cA$ into $B$ by $\widetilde\Phi_\Pi$. The map $\Pi\mapsto \widetilde\Phi_\Pi$ gives then a bijection between representations of $\cA$ in $B$ and $*$-homomorphisms from $\csp{\cA}/\cN_\cA$ into $B$. 

Now, for any $g=\sum_{s\in S}a_s\delta_s \in \csp{\cA}$ and any  representation $\Pi$ of $\cA$ in a $C^*$-algebra, we have
\[ \|\widetilde\Phi_\Pi(g+ \cN_\cA)\| = \|\Phi_\Pi(g)\| \leq \|g\|_{\rm u} \,.\]
It follows that 
if we define
\[ \|g+ \cN_\cA\|_{*}:=\sup_\Psi \big\{ \|\Psi\big(g+ \cN_\cA\big)\|\}\] where the supremum is taken over
 all $*$-homomorphisms $\Psi$ from $\csp{\cA}/\cN_\cA$ into a $C^*$-algebra, we get
 \begin{align*} 
 \|g+ \cN_\cA\|_{*}&= \sup\big\{ \|\Phi_\Pi(g)\| : 
\Pi \ \text{is a representation of $\cA$ in some $C^*$-algebra}\big\} \\
&\leq \|g\|_{\rm u} \,,
\end{align*}
so $\|\cdot\|_{*}$ gives a $C^*$-seminorm  on $\csp{\cA}/\cN_\cA$.  The  \emph{full $($Exel$)$ cross sectional $C^*$-algebra $C^*(\cA)$} is then defined  as the Hausdorff completion of $\csp{\cA}/\cN_\cA$ w.r.t.~to this seminorm. 

Letting $Q_{\cA}:\csp{\cA}\to \csp{\cA}/\cN_\cA$ denote the quotient map and $R_{\cA}:\csp{\cA}/\cN_\cA\to C^*(\cA)$ denote the canonical map, we get that $\iota_{\cA}:=R_{\cA}\,\circ \,Q_{\cA}$ is a contractive $*$-homomorphism from $\csp{\cA}$ into $C^*(\cA)$ having dense range. 
Thus, $\iota_{\cA}$ extends to a $*$-homomorphism $q_{\cA}$ from $C^{*}_{\rm KS}(\cA)$ onto $C^*(\cA)$ such that
\begin{equation}\label{fullcross}
C^*(\cA) \simeq \,C^{*}_{\rm KS}(\cA)\, / \,{\rm Ker} \, q_{\cA}\,.
\end{equation}
 Now, for each $s\in S$, define $\pi^{\cA}_s: A_s \to C^*(\cA)$ by \[\pi^{\cA}_s = q_{\cA}\, \circ \, \pi^{0}_s \ (=\iota_{\cA}\, \circ \, \pi^{0}_s)\,.\] Then one checks (cf.~\cite[Proposition 3.12]{exel11} and the proof of \cite[Proposition 3.13]{exel11}) that \[\Pi^{\cA}:=\big\{\pi^{\cA}_s\big\}_{s\in S}\] is a representation of $\cA$ in $C^*(\cA)$ satisfying the following universal property: given any representation $\Pi=\{\pi_s\}_{s\in S}$ of $\cA$ in a $C^*$-algebra $B$, there exists a unique $*$-homomorphism $\Psi_\Pi: C^*(\cA)\to B$ such that 
\,$\Psi_\Pi \circ\, \pi^{\cA}_s = \pi_s$ for all $s\in S$. It  follows immediately that $\Phi_\Pi = \Psi_\Pi \, \circ \,q_{\cA}$ for every such representation $\Pi$. 

The ideal $\,{\rm Ker} \, q_{\cA}$ has a natural description in terms of $\cN_\cA$. Indeed, letting $\cM_\cA$ denote the ideal of $C^{*}_{\rm KS}(\cA)$ given by $\cM_\cA:=\overline{\cN_\cA}^{\,\|\cdot\|_{\rm u}}$, we have \[\,{\rm Ker} \, q_{\cA} = \cM_\cA.\] 
 To prove this, we first note that since $q_{\cA}(\cN_\cA) = \iota_{\cA}(\cN_\cA) = \{0\} $, we have \[\cM_\cA \subset  \,{\rm Ker} \, q_{\cA}\,.\] Next, let $s\in S$ and
define $\omega_s:A_s \to C^{*}_{\rm KS}(\cA)/\cM_\cA$ by 
\[\omega_s (a_s) := \pi^0_s(a_s) + \cM_\cA \] for every $a_s \in A_s$. As $\cM_\cA$  contains  
 $\cN_\cA$, one easily verifies that  $\Omega=\{\omega_s\}_{s\in S}$ is  a representation of $\cA$ in $C^{*}_{\rm KS}(\cA)/\cM_\cA$.
The associated $*$-homomorphism $\Phi_\Omega$ from $C^{*}_{\rm KS}(\cA)$ into $C^{*}_{\rm KS}(\cA)/\cM_\cA$ is then nothing but the quotient map. Since $\Phi_\Omega = \Psi_\Omega \, \circ \,q_{\cA}$, it follows that  
\[{\rm Ker} \, q_{\cA} \subset {\rm Ker}\, \Phi_\Omega = \cM_\cA\,.\]
Thus we get  $\,{\rm Ker} \, q_{\cA} = \cM_\cA$, as desired. It follows that  
\begin{equation}\label{fullcross2}
C^*(\cA) \simeq \,C^{*}_{\rm KS}(\cA)\, / \cM_\cA\,.
\end{equation}

\subsection{}\label{sec:exelreduced}

In \cite{exel11}, Exel also constructs the reduced cross sectional $C^*$-algebra $C_r^{*}(\cA)$  of a Fell bundle $\cA = (\{A_s\}_{s\in S}, \{j_{t,s}\}_{s,t\in S,\, s\leq t})$. His construction, which is somewhat involved, may be summarized as follows. 

 Consider first $e\in E$ and $s\in S$ such that $e\leq s$. Then $j_{s,e}$ gives an isometric embedding of $A_e$ into $A_s$, so one may view $A_e$ as a subspace of $A_s$. Let $\varphi_e$ be 
a continuous linear functional on $A_e$. Exel shows in  \cite[Proposition 6.1]{exel11} that  $\varphi_e$ extends to a continuous linear functional $\tilde{\varphi}_e^s$ on $A_s$ satisfying  $\|\tilde{\varphi}_e^s\|=\|\varphi_e\|$ and  
\[\tilde{\varphi}_e^s(x) = \lim_i \varphi_e(xu_i) =  \lim_i \varphi_e(u_ix)=\lim_i \varphi_e(u_ixu_i)\]
for every approximate unit $\{u_i\}_i$ for $A_e$ and every $x \in A_s$.

Next, let $e\in E$ and let $\varphi_e$ be a state on $A_e$. Define $\tilde{\varphi}_e$ on $\csp{\cA}$ by
\[
 \tilde{\varphi}_e\left(\sum_{s\in S}a_s\delta_s\right)=\sum_{s\in S, \,s\geq e}\tilde{\varphi}_e^s(a_s)\,.
\]
Then, as shown in \cite[Proposition 6.9]{exel11}, $\tilde{\varphi}_e$ 
is a state on $\csp{\cA}$ when $\csp{\cA}$ is considered as a normed $*$-algebra with respect to the norm $\|g\|_1= \sum_{s\in S}\|a_s\|$ for $g =\sum_s a_s\delta_s \in \csp{\cA}$.  

\smallskip Now, let $\cE=\{A_e\}_{e\in E}$ denote the restriction of $\cA$ to the semilattice $E$, let $\Pi^{\cE}= \{\pi_e^{\cE}\}_{e\in E}$ denote the universal representation of $\cE$ in $C^*(\cE)$, and fix a pure state $\varphi$ on $C^*(\cE)$. For each $e \in E$, one has that $\pi^{\cE}_e(A_e)$ is an ideal of $C^*(\cE)$, and $\varphi_e:=\varphi\circ\pi_e^{\cE}$ is a state on $A_e$ as long as $\varphi_e \neq 0$, that is,~whenever $\varphi$ does not vanish on $\pi^{\cE}_e(A_e)$. 
Moreover, \cite[Proposition 7.4]{exel11} says that 
there exists a positive linear functional $
\tilde{\varphi}$ on $\csp{\cA}$ such that:
\begin{enumerate}
 \item For every $s\in S$ and $a_s\in A_s$, one has that
 \[
  \tilde{\varphi}(a_s\delta_s)=\begin{cases}
                                \tilde{\varphi}^s_e(a_s) & \text{if there exists 
                                  $e\in E$ such that $\varphi_e\neq 0$ and $e\leq s$,}\\
                                \hspace{2.5ex} 0                        & \text {otherwise,}
                               \end{cases}
 \]
 \item For every $e\in E$ and every $a_e\in A_e$ one has that
 \[
  \tilde{\varphi}(a_e\delta_e)= \varphi_e(a_e) = \varphi(\pi^{\cE}_e(a_e)),
 \]
 \item $\|\tilde{\varphi}\|\leq \|\varphi\|$,
 
 \smallskip \item $\tilde{\varphi}$ vanishes on the ideal $\cN_\cA$.
\end{enumerate}
For later use we note that if $S$ is $E^*$-unitary, then (i) and (ii) together just say that
\begin{equation}\label{eq:stateextension}
 \tilde{\varphi}(a_s\delta_s)=\begin{cases}
                                \varphi(\pi^{\cE}_s(a_s)) & \text{if $s\in E$,}\\
                               \hspace{4ex} 0                     & \text{ otherwise.}
                              \end{cases}
\end{equation}
Let $H_{\tilde{\varphi}}$ be the Hilbert space completion of $\csp{\cA}$ with respect to the pre-inner-product given by
\[
 \langle g,h \rangle_{\tilde{\varphi}}=\tilde{\varphi}(h^*g) \quad \text{for} \, \, g,h\in\csp{\cA},
\]
and let $h\mapsto\widehat{h}$ denote the canonical map $\csp{\cA}\to H_{\tilde{\varphi}}$.
The GNS representation of $\tilde{\varphi}$, 
which is defined in the usual way by 
\[
 \Upsilon_{\tilde{\varphi}}(g)\widehat{h}=\widehat{gh} \quad \text{for} \, \, g,h\in\csp{\cA},
\]
gives a $*$-representation of $\csp{\cA}$ on $H_{\tilde{\varphi}}$. 

Exel's reduced cross sectional $C^*$-algebra  $C^*_r(\cA)$ is then defined as the Hausdorff completion of $\csp{\cA}$ with respect to the $C^*$-seminorm
given by
\[
 \|g\|'_{\rm r}=
 \sup_{\varphi}\|\Upsilon_{\tilde{\varphi}}(g)\|
\]
where the supremum is taken over the set $\mathcal{P}(C^*(\cE))$ consisting of all pure states of $C^*(\cE)$. Note that the kernel of $ \Upsilon_{\tilde{\varphi}}$ is given by \[{\rm Ker} \Upsilon_{\tilde{\varphi}}=\{g\in\csp{\cA}:  \tilde{\varphi}(h^*gh') = 0 \text{ for all } h, h'\in \csp{\cA} \}.\] So if  $\mathcal{K}_\cA:= \{ g\in\csp{\cA}:  \|g\|'_{\rm r}=0\}$, then 
\[\mathcal{K}_\cA= \bigcap_{\varphi\in \mathcal{P}(C^*(\cE))} \, {\rm Ker} \Upsilon_{\tilde{\varphi}}\,,\]
 and $C^*_r(\cA)$ is the completion of $\csp{\cA}/\mathcal{K_\cA}$ with respect to the norm \[ \|g+ \mathcal{K}_\cA\|''_{\rm r} :=  \|g\|'_{\rm r}.\]
 Letting $\iota^{\rm red}_\cA: \csp{\cA} \to C^*_r(\cA)$ denote the canonical $*$-homomorphism, one gets 
 the left regular representation 
  $\Pi^{\rm red} = \{ \pi_s^{\rm red}\}_{s\in S}$ 
 of $\cA$ in $C^*_r(\cA)$ by setting 
  \[\pi_s^{\rm red} 
 =  \iota^{\rm red}_\cA \circ \pi^0_s\] for each $s\in S$. The associated $*$-homomorphism 
 $\Psi_{\Pi^{\rm red}}: 
 C^*(\cA) \to C^*_r(\cA)$ is then surjective  (cf.~\cite[Proposition 8.6]{exel11}).

\section{The left regular representation of $C_{\rm KS}^{*}(\cA)$}\label{leftregrep}

Let $\cA$ be a Fell-bundle over $S$.  In this section we will describe how one may define the left regular representation  $\Phi_\Lambda$ of $\csp{\cA}$ in a certain $C^*$-algebra $B$ naturally associated with $\cA$, and show that $\Phi_\Lambda$ is injective. We will first construct the left regular pre-representation $\Lambda$ of $\cA$ in $B$. The associated $*$-homomorphim $\Phi_\Lambda$ from $C_{\rm KS}^{*}(\cA)$ into $B$ will then give the left regular representation of $C_{\rm KS}^{*}(\cA)$.

\subsection{}\label{WagnerP}
We begin by recalling some notation and a few facts 
that will be useful 
in our construction.
 For each $u \in S$,  we set 
\[D(u) = \{s\in S:ss^*\leq u^*u\}\,,\] 
so  $D(u^*) = \{v\in S:vv^*\leq uu^*\}$, and for each $e\in E$, we set \[S_e=\{s\in S: s^*s=e\}\,.\] 
The Wagner-Preston theorem (and its proof), see for example \cite[Proposition 2.1.3]{paterson},  says that for each $u\in S$, the map $\gamma_u: D(u) \to D(u^*)$ given by $\gamma_u(s) = us$ is a bijection, with inverse given by $\gamma_{u^*}:D(u^*)\to D(u)$. Moreover, it says that the map $\gamma: u\mapsto \gamma_u$ is an injective homomorphism from $S$ into $\mathcal{I}(S)$. A part of the last statement is that for $u_1, u_2, s \in S$, we have
\begin{equation}\label{domain}
s \in D(u_1u_2)\ \text{if and only if} \ s \in D(u_2) \ \text {and} \ u_2s \in D(u_1)\,.
\end{equation}

\medskip 
Consider $u\in S$ and assume $s \in S_e \cap D(u)$ for some $e \in E$. Then we have
\[(us)^*us = s^*u^*us = s^*u^*u\,ss^*s = s^*ss^*s = s^*s = e\,\]
so $us \in S_e \cap D(u^*)$. Hence, if  $v\in S_e \cap D(u^*)$, then $u^*v \in S_e \cap D(u)$. 
It follows that the map
$s\mapsto us $ gives a bijection from $S_e \cap D(u) $ onto $ S_e \cap D(u^*)$, with inverse given by $v \mapsto u^*v$ for $ v\in S_e \cap  D(u^*)$. 

\subsection{} Let now $\cA=\{A_s\}_{s\in S}$ be a Fell bundle over $S$. 
Given $e\in E$, set
\begin{align*}
X_e=\Big\{\xi \in \prod_{s\in S_e}A_s
:\sum_{s\in S_e}\xi(s)^*\cdot\,\xi(s)\mbox{ is norm convergent in} \, \, A_e\Big\}\,.
\end{align*}
Note that the sum $\sum_{s\in S_e}\xi(s)^*\cdot\,\xi(s)$ makes sense since \[\xi(s)^*\cdot\xi(s)\in A_{s^*s}=A_e\] for each $s\in S_e$. Proceeding in the same way as for the direct sum of a family of right Hilbert $C^*$-modules over the same $C^*$-algebra \cite{lance95}, it is not difficult to check that  $X_e$ is a subspace of the product vector space $\prod_{s\in S_e}A_s$, which becomes a right Hilbert $A_e$-module  with respect to the operations
\begin{align*}
(\xi\cdot a)(s)&=\xi(s)\cdot a \, \in A_{se}=A_s, \\
\langle \xi,\eta \rangle_e&=\sum_{s\in S_e}\xi(s)^*\cdot\eta(s) \, \in A_e, 
\end{align*}
for $\xi, \eta \in X_e, \, a\in A_e$ and $ s\in S_e$. 

\smallskip Consider $e\in E$ and $u\in S$. For $a_u\in A_u$, we let
\[
\lambda_{e,u}(a_u):X_e\to X_e
\]
be the linear operator defined by
\[
\big(\lambda_{e,u}(a_u)\xi\big)(v)=\begin{cases}
a_u\cdot\xi(u^*v) & \mbox{ if } v\in D(u^*),\\
\hspace{4ex} 0                 & \mbox{ otherwise.}
\end{cases}
\]
for  $\xi \in X_e$ and $v\in S_e$. To see that $\lambda_{e,u}(a_u)$ is well defined, let $\xi \in X_e$.  
If $v\in S_e \cap D(u^*)$, 
then $u^*v\in S_e$, and
$\xi(u^*v) \in A_{u^*v}$, so we get 
\[a_u\cdot\xi(u^*v) \in A_u\cdot A_{u^*v}\subset A_{uu^*v}=A_{uu^*vv^*v}=A_{vv^*v}=A_v\,.\]
Thus we see that $\lambda_{e,u}(a_u)\xi$ lies in $\prod_{v\in S_e} A_v$. 
Moreover, if $v \in \,S_e \,\cap \,D(u^*)$, then one readily verifies that the map $b \mapsto a_u \cdot b$ is an adjointable linear map from $A_{u^*v}$ into $A_v$ (with adjoint map $c \mapsto a_u^*\cdot c$); thus, using \cite[Proposition 1.2]{lance95}, we get
\[\big (a_u\cdot \xi(u^*v)\big)^*\cdot \big(a_u\cdot \xi(u^*v)\big)\, \leq \, \|a_u\|^2\ \xi(u^*v)^*\cdot  \xi(u^*v)\,.\]
Now, since $\xi \in X_e$, the sum 
\[\sum_{v \in \,S_e \,\cap \,D(u^*)}  \xi(u^*v)^*\cdot  \xi(u^*v)\] is norm-convergent in $A_e$, and it follows 
that
\[\sum_{v\in S_e}\big(\lambda_{e,u}(a_u)\xi\big)(v)^*\cdot \big(\lambda_{e,u}(a_u)\xi\big)(v)
\,  = \sum_{v \in \,S_e \,\cap \,D(u^*)} \big (a_u\cdot \xi(u^*v)\big)^*\cdot \big(a_u\cdot \xi(u^*v)\big)\]
is also norm-convergent in $A_e$. Thus, $\lambda_{e,u}(a_u)\xi \in X_e$, as desired.

\medskip Next, we show that $\lambda_{e,u}(a_u)\in\lin{X_e}$. 
For $\xi, \eta \in X_e$, we have
\begin{align*}
\langle \lambda_{e,u}(a_u)\xi,\eta \rangle_e &= \sum_{v\in S_e}(\lambda_{e,u}(a_u)\xi)(v)^*\cdot\eta(v) \\
&= \sum_{v\in  S_e \cap D(u^*)}(a_u\cdot\xi(u^*v))^*\cdot\eta(v)\\
&= \sum_{v\in S_e \cap D(u^*) }\xi(u^*v)^*\cdot a_u^*\cdot\eta(v)\\
&= \sum_{s\in S_e \cap D(u) }\xi(s)^*\cdot a_u^*\cdot\eta(us)\\
&= \sum_{s\in S_e}\xi(s)^*\cdot(\lambda_{e,u^*}(a_u^*)\eta)(s) \\
&=\langle \xi,\lambda_{e,u^*}(a_u^*)\eta \rangle_e\,,
\end{align*}
where we have used that the map $v \mapsto u^*v$ is a bijection from $S_e \cap  D(u^*)$ onto $S_e \cap  D(u)$.
This shows that $\lambda_{e,u}(a_u)$
is an adjointable operator on $X_e$, with adjoint given by
\begin{equation}\label{adjoint}
\lambda_{e,u}(a_u)^*=\lambda_{e,u^*}(a_u^*).
\end{equation}

\smallskip Thus we get a map $\lambda_{e,u}:A_u\to\lin{X_e}$ for each $e\in E$ and each $u\in S$. 
For each $e \in E$ we set 
\[
\Lambda^e :=\{\lambda_{e,u}\}_{u\in S}\,.
\]
To show that $\Lambda^e$ is a pre-representation of $\cA$ in $\lin{X_e}$, in view of (\ref{adjoint}), we only have to show that for $u,u'\in S$, $a\in A_u$ and $a'\in A_{u'}$, we have
\begin{equation}\label{rep-prop}
\lambda_{e,uu'}(a\cdot a')=\lambda_{e,u}(a)\,\lambda_{e,u'}(a').
\end{equation}

To prove this, consider $\xi \in X_e$ and $v\in S_e$. Then 
\begin{align*}
\big(\lambda_{e,uu'}(a\cdot a')\xi\big)(v)&=\begin{cases}
a\cdot a'\cdot \xi((uu')^*v) & \mbox{ if } v\in D\big((uu')^*\big),\\
\hspace{6ex}0 & \mbox{ otherwise,}
\end{cases}\\
&=\begin{cases}
a\cdot a'\cdot \xi(u'^*u^*v) & \mbox{ if } v\in D\big(u'^*u^*\big),\\
\hspace{6ex}0 & \mbox{ otherwise,}
\end{cases}
\end{align*}
while
\[
\big(\lambda_{e,u}(a)\lambda_{e,u'}(a')\xi\big)(v)=\begin{cases}
a\cdot(\lambda_{e,u'}(a')\xi)(u^*v) & \mbox{ if } v\in D(u^*),\\
\hspace{8ex}0 & \mbox{ otherwise}
\end{cases}
\]
\[=\begin{cases}
a\cdot a'\cdot\xi(u'^*u^*v) & \mbox{ if } v\in D(u^*) \mbox{ and } u^*v \in D(u'^*),\\
\hspace{8ex}0 & \mbox{ otherwise.}
\end{cases}
\] 
Now, using (\ref{domain}) with $u_1 = u'^* $ and $u_2= u^*$ gives that $v\in D\big(u'^*u^*\big)$ if and only if $v\in D(u^*) \mbox{ and } u^*v \in D(u'^*)$, 
so we see that  \[\big(\lambda_{e,uu'}(a\cdot a')\xi\big)(v) = \big(\lambda_{e,u}(a)\lambda_{e,u'}(a')\xi\big)(v)\,.\] It follows that (\ref{rep-prop}) holds, as desired. 

\medskip We can now form the product pre-representation $\Lambda = \prod_{e\in E} \Lambda^e$ of $\cA$ in the product $C^*$-algebra $B:=\prod_{e\in E}\lin{X_e}$. It is natural to call $\Lambda$ the \emph{left regular pre-representation} of $\cA$ in $B$. It is given by $\Lambda =\{\lambda_u\}_{u\in S}$, 
where $\lambda_u:A_u\to B$ is defined by
\[
 \lambda_u(a_u)=\Big(\lambda_{e,u}(a_u)\Big)_{e\in E}
\]
for $u\in S$ and $a_u \in S$.
The associated  $*$-homomorphism $\Phi_\Lambda:\csp{\cA}\to B$ (resp.~$C^*_{\rm KS}(\cA)\to B$),
which satisfies
\[
\Phi_\Lambda\Big(\sum_{u\in S}a_u\delta_u\Big)=
 \ \left(\sum_{u\in S}\lambda_{e,u}(a_u)\right)_{e\in E},
\]
will be called 
the \emph{left regular representation of} $\csp{\cA}$ (resp.~$C^*_{\rm KS}(\cA)$) in $B$. Note that 
$\Phi_\Lambda=\prod_{e\in E}\Phi_{\Lambda^e}$, since
\[\big(\prod_{e\in E}\Phi_{\Lambda^e}\big)\Big(\sum_{u\in S}a_u\delta_u\Big)= \left(\Phi_{\Lambda^e}\Big(\sum_{u\in S}a_u\delta_u\Big) \right)_{e\in E} = \ \left(\sum_{u\in S}\lambda_{e,u}(a_u)\right)_{e\in E}.\]

\subsection{}
Our aim is to show that $\Phi_\Lambda$ is injective on $\csp{\cA}$ (cf.~Theorem \ref{theo:wording}). The following lemma will be crucial. 

\begin{lemma}\label{lem:wordinghamsub}
 Assume
 that $g=\sum_{u\in S}a_u\delta_u\in\csp{\cA}$ satisfies $\Phi_\Lambda(g)=0$  
and let $e,f\in E$ with $e\leq f$. 
 
 \smallskip Then, for each $t\in S$ and each $b\in A_e$,
 we have
\[
\sum_{u\in S, \,f\leq u^*u, \,ue=t } a_u\cdot b = 0\,.
\]
\end{lemma}
Note that here (and elsewhere), we use the convention that a sum over an empty index set is equal to $0$.
\begin{proof}
For  each $s\in S$ and $a\in A_s$, we will let $a\odot\varepsilon_{s}$ denote the element  of $X_{s^*s}$  given for each $t\in S_{s^*s}$ by 
\[(a\odot\varepsilon_{s})(t) = \begin{cases} a & \ \text{if} \ t=s,\\
0 & \ \text{if} \ t \neq s\,.
\end{cases}
\]
For every $v\in E$, we set $g_v=\sum_{u\in S_v} a_u\delta_u$. Then $g_v\in\csp{\cA}$, $g_v=0$ for all but finitely many $v$ in $E$, and \[g=\sum_{v\in E}g_v\,.\] Moreover, 
\begin{equation}\label{eq:wordingham1}
0=\Phi_\Lambda(g)=\sum_{v\in E}\Phi_\Lambda(g_v)=\left(\sum_{v\in E}\sum_{u\in S_v}\lambda_{p,u}(a_u)\right)_{p\in E}
\end{equation}

\medskip \noindent Now, consider $v\in E$, $u\in S_v$, $a\in A_u$ and $a'\in A_f$. 

\medskip Note first  that $a'\odot\varepsilon_f \in X_{f^*f}=X_f$. Moreover, if $f\leq v$, then \[(uf)^*uf = fu^*uf= fvf = f\,,\] so $uf \in S_f$ and $(a\cdot a')\odot\varepsilon_{uf} \in X_{(uf)^*uf} = X_f$.
We claim that \begin{equation}\label{regfund}
\lambda_{f,u}(a)(a'\odot\varepsilon_f)=\begin{cases}
(a\cdot a')\odot\varepsilon_{uf}&\mbox{if }f\leq v\,,\\
\hspace{6ex}0&\mbox{otherwise.}
\end{cases}
\end{equation}
To prove this claim, let $t\in S_{f}$. 
Then we have 
 $t \in D(u^*) $, that is, $tt^*\leq uu^*$, if and only if $f=t^*t\leq u^*u=v$. As
\[
\big(\lambda_{f,u}(a)(a'\odot\varepsilon_f)\big)(t)=\begin{cases}
a\cdot (a'\odot\varepsilon_{f})(u^*t)&\mbox{if } t\in D(u^*)\\
\hspace{6ex} 0&\mbox{otherwise,}
\end{cases}
\]
we see that $\lambda_{f,u}(a)(a'\odot\varepsilon_f)=0 $ when $f\not\leq v$. 

\smallskip If $f\leq v$, thus
 $t\in S_f \cap  D(u^*)$, then we have  $u^*t \in S_f \cap D(u)$, with $u^*t = f$ if and only $uf =t$ (cf.~ \ref{WagnerP}), so we get  
\begin{align*}\big(\lambda_{f,u}(a)(a'\odot\varepsilon_f)\big)(t) &=  a\cdot (a'\odot\varepsilon_{f})(u^*t)\\
&=\begin{cases}
a\cdot a' &\mbox{if } t= uf \\
\hspace{2ex} 0&\mbox{otherwise}
\end{cases}
\,\, =
\big((a\cdot a')\odot\varepsilon_{uf}\big)(t)\,.
\end{align*}
We have thus shown that $\lambda_{f,u}(a)(a'\odot\varepsilon_f)= (a\cdot a')\odot\varepsilon_{uf}$ whenever $f\leq v$, and this finishes the proof of (\ref{regfund}).

\medskip Let now $b \in A_e$. By the Cohen-Hewitt factorization theorem \cite[Theorem 32.22]{hewitt_ross70} we can write $b$ as a product 
$b=cd$ where $c,d\in A_e$. As $e\leq f$, we get from (\ref{a-e-f}) that
 \begin{equation}\label{c-d}
 j_{f,e}(c)\cdot d= cd =b\,.
 \end{equation} 

For each $v\in E$  we get from (\ref{regfund}) that
\[
\sum_{u\in S_v}\lambda_{f,u}(a_u)(j_{f,e}(c)\odot\varepsilon_f)=\begin{cases}
\sum_{u\in S_v}(a_u\cdot j_{f,e}(c))\odot\varepsilon_{uf}&\mbox{if }f\leq v\,,\\
\hspace{12ex}0&\mbox{otherwise.}
\end{cases}
\]
Using  \eqref{eq:wordingham1} it then follows that
\[
0=\sum_{v\in E}\sum_{u\in S_v}\lambda_{f,u}(a_u)(j_{f,e}(c)\odot\varepsilon_f)=\sum_{\{v\in E:f\leq v\}}\sum_{u\in S_v}(a_u\cdot j_{f,e}(c))\odot\varepsilon_{uf}\,.
\]
By looking at individual coefficients we can then conclude that in $\csp{\cA}$,
\begin{equation}\label{eq:wordingham2}
0=\sum_{\{v\in E:f\leq v\}}\sum_{u\in S_v}(a_u\cdot j_{f,e}(c))\,\delta_{uf}
\end{equation}
Since $e\leq f$ we get from  \eqref{eq:wordingham2} and (\ref{c-d}) that 
\begin{equation*}
0=\left(\sum_{\{v\in E:f\leq v\}}\sum_{u\in S_v}(a_u\cdot j_{f,e}(c))\,\delta_{uf}\right)\big(d\delta_e\big)=\sum_{\{v\in E:f\leq v\}}\sum_{u\in S_v}(a_u\cdot b)\,\delta_{ue}\,.
\end{equation*}
We see that given $t\in S$, the $t$-coefficient of the sum on the right hand side of the above equation is
\[
\sum_{u\in S, \,f\leq u^*u, ue=t}\, a_u\cdot b\,,
\]
which must then be equal to $0$.
\end{proof}
We will need another lemma.
Let $F$ be a semilattice and $A$ be a Banach space. As usual, we will denote the dual space of $A$, consisting of all continuous linear functionals on $A$, by $A^*$.
We let $\csp{F,A}$ denote the vector space of all 
finitely supported functions from $F$ to $A$. We will describe an element of $\csp{F,A}$ as a formal sum
$\sum_{f\in F}a_f\delta_f$ where each $a_f\in A$ and $a_f=0$ for all but finitely many $f$ in $F$.
Given $\psi\in A^*$ and $e\in F$, we define $\theta_{\psi,e}:\csp{F,A}\to\mathbb{C}$ to be the linear functional given by
\[
\theta_{\psi,e}\Big(\sum_{f\in F}a_f\delta_f\Big)= \sum_{f \in F, \,f\geq e} \psi(a_f)
\]

\begin{lemma}\label{lem:separation}
Let $A$ be a Banach space and $F$ be a semilattice. Then the set $\{\theta_{\psi,e}:e\in F,\psi\in A^*\}$ separates the elements of $\csp{F,A}$.
\end{lemma}
\begin{proof}
Suppose $\sum_{f\in F}a_f\delta_f\neq 0$. Since $a_f=0$ for all but finitely many $f$ in $F$, we can choose $e\in F$ such that $a_e\neq 0$ and $a_f=0$ for all $f\in F\setminus \{e\}$ satisfying $f\geq e$. We may then pick $\psi\in A^*$ such that $\psi(a_e)\neq 0$, and this gives
\[
 \theta_{\psi,e}\left(\sum_{f\in F} a_f\delta_f\right)=\sum_{f\in F,\,f\geq e}\psi(a_f)=\psi(a_e)\neq 0.
\]
\end{proof}

The following theorem is a generalization of Wordingham's theorem \cite[Theorem 2.1.1]{paterson}, and our proof follows the pattern of Wordingham's original proof.
\begin{theorem}\label{theo:wording}
Let $\cA=\{A_s\}_{s\in S}$ be a Fell bundle over an inverse semigroup $S$. Then the left regular representation of $\csp{\cA}$ is injective.
\end{theorem}
\begin{proof}
Let $g\in\csp{\cA}$ and express $g$ as a sum $g=\sum_{u\in S}a_u\delta_u$, where $\mathrm{supp}(g)=\{u\in S : a_u\neq 0\}$ is finite. 
Assume $\Phi_\Lambda(g)=0$. We want to show that $a_t=0$ for each $t\in S$. Since $g=\sum_{e\in E}\sum_{u\in S_e}a_u\delta_u$ it is sufficient to show that for any $e\in E$, $a_t=0$ when $t\in S_e$. 

Fix $e\in E$ and consider $t\in S$. Let $F$ be the subsemilattice of $E$ given by \[F=\{v\in E:e\leq v\}\,.\] Also, let $f\in F$ and $b\in A_e$. For each $v\in F$ set
\begin{align*}
\beta^t_v&=\sum_{u\in S, \, v=u^*u, \,ue=t} a_u\cdot b \,\, \, \in A_t\,,\\
\beta^t&=\sum_{v\in F}\,\beta^t_v\, \delta_v \,\,  \in \csp{F,A_t}\,.
\end{align*}
Note that $\beta^t$ has finite support since $\beta^t_v=0$ if $v\notin\{u^*u:u\in\mathrm{supp}(g)\}$, and $\mathrm{supp}(g)$ is finite. Now, for each $\psi\in (A_t)^*$, we get from Lemma~\ref{lem:wordinghamsub} that
\begin{align*}
\theta_{\psi,f}\big(\beta^t\big)&=\sum_{v\in F,\,f\leq v}\,\psi\big(\beta^t_v\big)\\
&=\psi\left(\sum_{u\in S,\, f\leq u^*u,\,ue=t} a_u\cdot b\right)=0\,.
\end{align*}
 Then $\beta^t=0$ by Lemma~\ref{lem:separation}, so $\beta^t_v=0$ for each $v\in F$. In particular, since $e\in F$, we get
\begin{equation}\label{beta-te}
\sum_{u\in S_e, \, ue=t} a_u\cdot b = \beta^t_e = 0\,.
\end{equation}
Assume now that $t\in S_e$. If $u\in S_e$ satisfies that $ ue=t$, then we have $ut^*t=t$ and $u^*u=t^*t$, which together imply that $u=t$. So (\ref{beta-te}) gives that $a_t\cdot b=0$. Choosing $b=a_t^*\cdot a_t\in A_{t^*t}=A_e$, we get $a_t\cdot a_t^*\cdot a_t=0$, so $(a_t^*\cdot a_t)^2 = a_t^*\cdot a_t\cdot a_t^*\cdot a_t=0$, hence $a_t^*\cdot a_t=0$, and axiom (x)
in the definition of a Fell bundle gives that $a_t=0$, as desired.
\end{proof}

\subsection{} 
We define the
\emph{reduced KS-cross sectional $C^*$-algebra  $\CrKS{\cA}$ of $\cA$} as the completion of $\csp{\cA}$ with respect to the norm $\|\cdot\|_r$ given by
\[\|g\|_{\rm r} :=\|\Phi_\Lambda(g)\| \] for $g \in \csp{\cA}$. Alternatively, we may consider $\CrKS{\cA}$ to be given as the norm-closure of $\Phi_\Lambda(\csp{\cA})$ in $B$, or, equivalently, as $\Phi_\Lambda(\CKS{\cA})$.

Recall that $\mathcal{N}_\cA$ denotes the two-sided selfadjoint ideal of $\csp{\cA}$ spanned by the set
\[
 \{a_s\delta_s-j_{t,s}(a_s)\delta_t:s,t\in S, s\leq t, a_s\in A_s\}\,.
\]
We define $\mathcal{I}_\cA$ to be the closure of $\mathcal{N}_\cA$ inside $\CrKS{\cA}$. In other words, we set \[\mathcal{I}_\cA= \overline{\Phi_\Lambda(\mathcal{N}_\cA)}\,.\] It is easy to check that $\mathcal{I}_\cA$ is an 
ideal of $\CrKS{\cA}$. Hence we may form the quotient $C^*$-algebra \[\CrAlt{\cA}:=\CrKS{\cA}/\mathcal{I}_\cA\,,\] which 
provides an alternative version of the reduced cross sectional $C^*$-algebra of $\cA$. We will let $q_\cA^r:\CrKS{\cA}\to \CrAlt{\cA}$ denote the quotient map. It is not clear 
whether $\CrAlt{\cA}$ is isomorphic to Exel's reduced $C^*$-algebra $C^*_r(\cA)$ (cf.~\ref{sec:exelreduced}). 
We will show in Section \ref{sec:comparison} that 
this is true under certain assumptions.

For $u\in S$ let $\lambda^{\rm alt}_u:A_u \to \CrAlt{\cA}$ be defined by 
\[\lambda^{\rm alt}_u(a_u) = \lambda_u(a_u) + \mathcal{I}_\cA\] for all $a_u \in A_u$. It is then almost immediate that  $\Lambda^{\rm alt}:=\{\lambda^{\rm alt}_u\}_{u\in S}$ is a representation of $\cA$ in 
$\CrAlt{\cA}$. 
Using the universal property of $C^*(\cA)$ we get a surjective $*$-homomorphism $\Psi_{\Lambda^{\rm alt}}$ from $C^*(\cA)$ onto  
$\CrAlt{\cA}$ satisfying \[\Psi_{\Lambda^{\rm alt}}\big(\pi_u^\cA(a_u)\big) = \lambda^{\rm alt}_u(a_u) = \lambda_u(a_u) + \mathcal{I}_\cA \] for all $u\in S$ and $a_u \in A_u$. Similarly, we get a surjective $*$-homomorphism $\Phi_{\Lambda^{\rm alt}}$ from $\CKS{\cA}$ onto  
$\CrAlt{\cA}$, which satisfies \[\Phi_{\Lambda^{\rm alt}} = q^r_\cA \circ \Phi_{\Lambda} = \Psi_{\Lambda^{\rm alt}}\circ q_\cA\,.\]

The following commutative diagram sums up the relationship between the various algebras and some of the $*$-homomorphisms defined so far.

\begin{center}
\begin{tikzpicture}[>=angle 90, scale=3, text height=1.5ex, text depth=0.25ex]
\node (CcA)    at (0,2) {$\csp{\cA}$};
\node (CKSA)   at (1,2) {$\CKS{\cA}$};
\node (CrKSA)  at (2,2) {$\CrKS{\cA}$};
\node (CcANA)  at (0,1) {$\csp{\cA}/\cN_\cA$};
\node (CA)     at (1,1) {$C^*(\cA)$};
\node (CraltA) at (2,1) {$\CrAlt{\cA}$};
\node (CrA)    at (1,0) {$C^*_r(\cA)$};
\path[right hook->,font=\scriptsize]
(CcA) edge (CKSA);
\path[->>,font=\scriptsize]
(CcA) edge node[auto] {$Q_\cA$} (CcANA)
(CKSA) edge node[auto] {$q_\cA$} (CA)
(CrKSA) edge node[auto] {$q_\cA^{\rm r}$} (CraltA)
(CKSA) edge node[above] {$\Phi_\Lambda$} (CrKSA)
(CKSA) edge node[sloped, above, yshift=1pt] {$\Phi_{\Lambda^{\rm alt}}$} (CraltA)
(CA) edge node[above, yshift=1pt] {$\Psi_{\Lambda^{\rm alt}}$} (CraltA)
(CA) edge node[auto] {$\Psi_{\Pi^{\rm red}}$} (CrA);
\path[->,font=\scriptsize]
(CcA) edge node[sloped, above] {$\iota_\cA$} (CA)
(CcANA) edge node[above] {$R_\cA$} (CA)
(CcA) edge[bend right=80] node[above, sloped, near end] {$\iota_\cA^{\rm red}$} (CrA);
\end{tikzpicture}
\end{center}

\section{Fell bundles over semilattices}\label{semilatt}

In this section we look at the case where $S=E$ is a semilattice, and consider a Fell bundle $\cE=\{A_e\}_{e\in E}$. 
Since $E_e=\{f\in E:f^*f=e\}=\{e\}$ for each $e\in E$, the Hilbert $A_e$-module $X_e$ that occurs 
in the definition of the pre-representation $\Lambda^e=\big(\lambda_{e,f}\big)_{f\in E}$ of $\cE$ in $\lin{X_e}$  is nothing but $A_e$ itself (with its standard structure).
Thus $\CrKS{\cE}$ can be viewed as a $C^*$-subalgebra of $\prod_{e\in E}\lin{A_e}$,
and for $e, f\in E$ and $a_f\in A_f$, $\lambda_{e,f}(a_f):A_e\to A_e$ is given by
\begin{equation}\label{eq:lamb-e-f}
 \lambda_{e,f}(a_f)b = \begin{cases}
                          a_f\cdot b & \text{ if } e\leq f,\\
                          0          & \text{ otherwise, }
                         \end{cases}   \quad \quad \text{for  all}\ b\in A_e.                    
\end{equation}
As before, let $\Phi_{\Lambda^e}:\csp{\cE}\to\lin{A_e}$ be the corresponding $*$-homomorphism given by
\[
\Phi_{\Lambda^e}\Big(\sum_{f\in E}a_f\delta_f\Big)=\sum_{f\in E}\lambda_{e,f}(a_f).
\]
Let $(a^i)$ be an approximate unit for $A_e$,  let $f\in E$ be such that $f\geq e$, and  let $a_f\in A_f$. Then for all $b\in A_e$ we have 
\[
 \lim_i \,(a^i \cdot a_f)\cdot b = \lim_i \,a^i \cdot (a_f \cdot b) = a_f\cdot b,
\]
and $\lim_i \, a_f \cdot (a^i\cdot b)=a_f\cdot b$. Thus,
\begin{equation}\label{eq:reducedrepformula}
\lambda_{e,f}(a_f)=\lim_i a^i\cdot a_f=\lim_i a_f\cdot a^i
\end{equation}
where the limits in equation \eqref{eq:reducedrepformula} are taken in the strict topology of $\lin{A_e}$.

\medskip We recall that a  \emph{character} (sometimes called a semicharacter) on $E$ is a non-zero homomorphism from $E$ into the semilattice $ \{0,1\}$,
see e.g.~\cite{paterson}.

\begin{lemma}\label{semichar}
 Let $\pi:\CKS{\cE}\to\mathcal{B}(H)$ be a non-zero irreducible representation of $\CKS{\cE}$ on a Hilbert space $H$. 
 For $e \in E$, let $p_e$ denote the orthogonal projection of $H$ onto the norm-closure of $\pi(A_e\delta_e)H$ in $H$. Then $p_e\in\{0,I_H\}$. Moreover,  the map $\widehat\pi:E\to \{0,1\}$ defined by 
\[\widehat\pi(e) = \begin{cases} 1& \quad \text{if} \ p_e = I_H, \\ 0& \quad \text{if} \ p_e = 0,
\end{cases}\]  
is a character on $E$.
\end{lemma}
\begin{proof}
 Let $(a^i)$ be an approximate unit for $A_e$. It is straightforward to check that $\pi(a^i\delta_e)$ converges strongly to $p_e$. 
 
 Let $f\in E$ be such that $f\leq e$ and let $a\in A_f$. Since $j_{e,f}(a^i\cdot a)=a^i\cdot j_{e,f}(a)$, and $a^i\cdot j_{e,f}(a)$ converges to $j_{e,f}(a)$ in norm, it follows, using that $j_{e,f}$ is isometric, that $a^i\cdot a$ converges to $a$ in norm. Hence, for any $\xi\in H$, we have
 \[
  p_e\pi(a\delta_f)\xi = \lim_i \pi(a^i\delta_e)\pi(a\delta_f)\xi =\lim_i \pi(a^i\cdot a\delta_f)\xi=\pi(a\delta_f)\xi.
 \]
  It follows that $p_fH\subset p_eH$, that is, $p_f \leq p_e$.
  
  Consider now $e'\in E$. Then $e'e \leq e$, so $p_{e'e}H\subset p_eH$. 
  
  Hence for $a\in A_{e'}$ we get 
 \begin{align*}
  \pi(a\delta_{e'})p_eH&=\pi(a\delta_{e'})\overline{\pi(A_e\delta_e)H}\\
  &\subset\overline{\pi(A_{e'}\cdot A_e\delta_{e'e})H}\\&\subset \overline{\pi(A_{e'e}\delta_{e'e})H}\\&=p_{e'e}H\subset p_eH.
 \end{align*}
  This implies that $p_eH$ is a closed invariant subspace for $\pi(\CKS{\cE})$, hence that $p_e\in\{0,I_H\}$ since $\pi$ is irreducible. 
  Moreover, for any $e,e'\in E$, $p_{e'}p_e=p_{e}p_{e'}$ is then  a projection, and, as seen above,  we have $p_{e'e}\leq p_e$, and similarly $p_{e'e}=p_{ee'} \leq p_{e'}$, so $p_{e'e}\leq p_{e'}p_{e}$. 
  On the other hand, for $a\in A_{e'}$, we know that $\pi(a\delta_{e'}) p_{e}H\subset p_{e'e}H$. Hence, using an approximate unit for $A_{e'}$, one easily deduces that $p_{e'}p_{e}\leq p_{e'e}$. Thus we get $p_{e'e} = p_{e'}p_{e}$. Since $\pi$ is non-zero,
it clearly follows that $\widehat{\pi}$ is a character on $E$.  
\end{proof}
Given a character $\psi$ on $E$, we set $F_\psi = \{e\in E:\psi(e)=1\}$. Then $F_\psi$ is an example of a filter in $E$ (and every filter on $E$ can be obtained this way), cf.~\cite{paterson}. We recall that a filter in $E$ is a nonempty subsemilattice $F$ of $E$ such that if $f\in E$ and $f\geq e$ for some $e\in F$, then $f\in F$. 
\begin{proposition}\label{prop:universalKS}
 We have $\CrKS{\cE}=\CKS{\cE}$.
\end{proposition}
\begin{proof}
 We only have to check that every irreducible representation of $\CKS{\cE}$ is dominated in norm by the left regular representation $\Phi_\Lambda$. More precisely,  it suffices to show that given a Hilbert space $H$ and a non-zero irreducible representation $\pi:\CKS{\cE}\to \mathcal{B}(H)$, we have \[\|\pi(g)\|\leq \|\Phi_\Lambda(g)\|\] for all $g\in \csp{\cE}$. Let $\widehat{\pi}$ be the character on $E$ described in Lemma \ref{semichar}, and let $F=F_{\widehat{\pi}}$ be the corresponding filter in $E$. Note that if $p_e$ is defined as in Lemma \ref{semichar}, then $p_e=I_H$ when $e\in F$, while $p_e = 0$ when $e\in E\setminus F$. Hence for all $e \in E\setminus F$   
and all $a_e\in A_e$ we have $\pi(a_e\delta_e)=0$. Indeed, letting $(a^i)$ be an approximate unit for $A_e$, we then have
 \[\pi(a_e\delta_e)\xi = \lim_i \pi(a_ea^i \delta_e)\xi = \lim_i \pi(a_e\delta_e)\pi(a^i\delta_e)\xi = \pi(a_e\delta_e)p_e\xi = 0\]
 for all $\xi \in H$.
Consider now $g =\sum_{u\in E} a_u\delta_u \in\csp{\cE}$. Using the observation we just made, we get
 \[\pi(g) =  \sum_{f\in K\cap F}\pi(a_f\delta_f),\]
 where $K:=\{u\in E:a_u\neq 0\}$ is finite. Since $F$ is a semilattice and $K\cap F$ is a finite subset of $F$, there exists some $e\in F$ such that $e\leq f$ for all $f\in K\cap F$. 
 Let $(a^i)$ be an approximate unit for $A_e$. Since the restriction $\pi_e$ of $\pi$ to $A_e\delta_e\simeq A_e$ is a non-degenerate representation of $A_e\delta_e$ on $H$, it may be extended to a representation $\overline{\pi_e}:\lin{A_e}\to\mathcal{B}(H)$
  (see for instance \cite[Theorem II.7.3.9]{blackadar}). 
  Moreover, if $(b^i\delta_e)$ is a net in $A_e\delta_e$ converging strictly to some $x\in\lin{A_e}$, 
  then $\pi_e(b^i\delta_e)$ converges strongly to $\overline{\pi_e}(x)$ 
  in $\mathcal{B}(H)$. Thus for any $\xi\in H$, using equations \eqref{eq:lamb-e-f} and \eqref{eq:reducedrepformula}, we get
 \begin{align*}
  \pi(g)\xi&=p_e\pi(g)\xi= \lim_i\pi(a^i\delta_e)\,\sum_{f\in K\cap F}\pi(a_f\delta_f)\xi\\
           &=\lim_i \sum_{f\in K \cap F} \pi\Big((a^i\cdot a_f)\delta_e\Big)\xi\\
           &=\sum_{f\in K\cap F} \overline{\pi_e}(\lambda_{e,f}(a_f))\xi\\
           &=\overline{\pi_e}(\Phi_{\Lambda^e}(g))\xi\,.
 \end{align*}
 It follows that $\pi(g)=\overline{\pi_e}(\Phi_{\Lambda^e}(g))$, so $\|\pi(g)\|\leq\|\Phi_{\Lambda^e}(g)\|\leq \|\Phi_\Lambda(g)\|$.
\end{proof}
Proposition \ref{prop:universalKS} implies that $\cM_\cE=\mathcal{I}_\cE$, hence that $C^*(\cE)=\CrAlt{\cE}$. Since it follows from \cite[Corollary 8.10]{exel11} that $C^*_r(\cE)\simeq C^*(\cE)$, we get:
\begin{corollary}
$ \CrAlt{\cE}=C^*(\cE)\simeq C^*_r(\cE).$
\end{corollary}
\section{Conditional Expectations onto the Diagonal}\label{cond-exp-diag}

Let $\cA=\{A_s\}_{s\in S}$ be a Fell bundle over $S$ and let $\cE=\{A_e\}_{e\in E}$ denote the Fell bundle  obtained by restricting $\cA$ to the semilattice $E$ of idempotents in $S$. Recall that $\CrKS{\cA}$ can be viewed as a $C^*$-subalgebra of $\prod_{e\in E}\lin{X_e}$ and that $\CrKS{\cE}$ can be viewed as a $C^*$-subalgebra of $\prod_{e\in E}\lin{A_e}$. When it is necessary to distinguish them, we will denote by $\Phi^\cA_{\Lambda}$ the left regular representation of $\csp{\cA}$ and by $\Phi^\cE_\Lambda$ the left regular representation of $\csp{\cE}$. Similarly, we will write $\{\lambda^\cA_{e,s}\}$ and $\{\lambda^\cE_{e,f}\}$ for the respective pre-representations. 

Let $e\in E$ and $s\in S_e$ (so that $s^*s=e$). Recalling that $A_s$ is a (right) $A_e$-module, we define an $A_e$-module map $\gamma_s:A_s\to X_e$ by $\gamma_s(a)=a\odot\varepsilon_s$, i.e.
\[
 \gamma_s(a)(t)=\begin{cases}
                   a & \mbox{ if }s=t\\
                   0   & \mbox{ otherwise.}
                  \end{cases}
\]
One readily checks that $\gamma_s$ is adjointable with adjoint given by $\gamma_s^*(\xi)=\xi(s)$ for every $\xi \in X_e$. Then $\gamma_s^*\gamma_s$ is clearly the identity map on $A_s$, so $\gamma_s$ is isometric. Moreover, we have $\gamma_s\gamma_s^*\xi = \xi(s)\odot \varepsilon_s$ for every $\xi \in X_e$, and it follows that $\sum_{s\in S_e}\gamma_s\gamma_s^*\xi=\xi$ for every $\xi\in X_e$, where the sum converges in the norm topology on $X_e$.

\begin{lemma}\label{lem:conditional0}
 Let $e\in E$, $s,t\in S_e$, $u\in S$ and $a\in A_u$. Then the map $\gamma_s^*\lambda_{e,u}(a)\gamma_t: A_t\to A_s$ is given by
 \[
  \big(\gamma_s^*\lambda_{e,u}(a)\gamma_t\big)(b)=\begin{cases}
                                          a\cdot b & \text{ if }u \geq st^*,\\
                                          0        & \text{ otherwise}
                                         \end{cases}
 \]
 for all $b\in A_t$.
\end{lemma}
\begin{proof}
 For $b\in A_t$ we have
 \begin{align*}
  \gamma_s^*\lambda_{e,u}(a)\gamma_t(b)&=\big(\lambda_{e,u}(a)\gamma_t(b)\big)(s)\\
  &=\begin{cases}
     a\cdot\gamma_t(b)(u^*s)   & \text{ if }s \in D(u^*),\\
     \hspace{6ex} 0                         & \text{ otherwise}
    \end{cases}\\
  &=\begin{cases}
     a\cdot b                  & \text{ if }u^*s=t\text{ and }ss^*\leq uu^*,\\
     \hspace{2ex} 0                         & \text{ otherwise.}
    \end{cases}
 \end{align*}
 Suppose first that $u\geq st^*$. Then $uu^*\geq (st^*)(st^*)^*=ss^*$ since $t^*t=s^*s$. Moreover, $u^*\geq ts^*$, so $ts^*=u^*(ts^*)^*(ts^*)=u^*ss^*$, hence $u^*s=ts^*s=tt^*t=t$. Conversely, if $ss^*\leq uu^*$ and $u^*s=t$, then \[u^*(ts^*)^*(ts^*)=u^*st^*ts^* = u^* uu^*ss^*ss^* =u^*ss^*=ts^*,\] so $u^*\geq ts^*$, hence $u\geq st^*$.
\end{proof}

\begin{lemma}\label{lem:factorize}
 Let $t\in S$ and $b\in A_t$. Then there exist $c\in A_{tt^*}$ and $d\in A_t$ such that  $b=c\cdot d$.
 \end{lemma}
\begin{proof}
Let $(u^i)$ be an approximate unit for  $A_{tt^*}$. Since $||u^i||\leq 1$ for all $i$, we get
 \begin{align*}
  ||u^i\cdot b-b||^2&=||(u_i\cdot b-b)(u_i\cdot b-b)^*||\\
  &=||u_i\cdot b\cdot b^*-b\cdot b^*+(u_i\cdot b\cdot b^*-b\cdot b^*)\cdot u_i||\\
  &\leq 2||u_i\cdot b\cdot b^*-b\cdot b^*||.
 \end{align*}
So $u^i\cdot b$ converges to $b$. Regarding $A_t$ as a left $A_{tt^*}\,$-module in the obvious way, we may then apply the Cohen-Hewitt factorization theorem \cite[Theorem~32.22]{hewitt_ross70} to $b$ and deduce that $b=c\cdot d$ for some $c\in A_{tt^*}$ and $d\in A_t$.
\end{proof}

\begin{lemma}\label{lem:fourierfactorize}
 Let $t\in S$ and $b\in A_t$. Let $b=c\cdot d$ be any factorization of $b$ with $c\in A_{tt^*}$ and $d\in A_t$. Let $(T_e)_{e\in E}\in\CrKS{\cA}$, and let $s\in S$ be such that $s^*s=t^*t$. Then
 \[
  (\gamma_s^*T_{t^*t}\gamma_t)(b)= (\gamma_{st^*}^*T_{tt^*}\gamma_{tt^*})(c)\cdot d.
 \]
\end{lemma}
\begin{proof}
Note first that the expression on the right-hand side is well-defined since $(st^*)^*(st^*)= t s^*st^*= tt^*$ as $s^*s =t^*t$.
By linearity and continuity, it suffices to prove that for any $u\in S$ and $a\in A_u$, we have
 \[
  (\gamma_s^*\lambda_{t^*t,u}(a)\gamma_t)(b)= (\gamma_{st^*}^*\lambda_{tt^*,u}(a)\gamma_{tt^*})(c)\cdot d.
 \]
 This follows immediately by applying Lemma~\ref{lem:conditional0} to both sides, and using that $st^*(tt^*)^*=st^*$.
\end{proof}

\begin{lemma}\label{lem:conditional1}
For any $e,f\in E$, and $a_f \in A_f$ we have
\[
 \gamma_e^*\lambda^\cA_{e,f}(a_f)\gamma_e = \lambda^\cE_{e,f}(a_f).
\]
Moreover, if $S$ is $E^*$-unitary and $A_0=\{0\}$ (if $S$ has a $0$-element), then for any $e\in E$, $u\in S$ and $a_u\in A_u$, we have
 \begin{equation}\label{gamma-e-u}
 \gamma_e^*\lambda^\cA_{e,u}(a_u)\gamma_e = \begin{cases}
                                         \lambda^\cE_{e,u}(a_u)\quad&\mbox{ if } u\in E,\\
                                         \hspace{3ex}0 & \mbox{ otherwise.}
                                        \end{cases}
\end{equation}
\end{lemma}
\begin{proof}
 We prove the second statement. The proof of the first statement follows from a small adjustment to the argument and is left to the reader. Assume that $S$ is $E^*$-unitary  and $A_0=\{0\}$ (if $S$ has a $0$-element). Let $b\in A_e$. Lemma~\ref{lem:conditional0} gives that
 \begin{equation}\label{gamma-e}
 \big(\gamma_e^*\lambda^\cA_{e,u}(a_u)\gamma_e\big)(b)= 
  \begin{cases}
    a_u\cdot b & \text{if } e\leq u,\\
    \hspace{2ex}0 & \text{otherwise. }
  \end{cases}
\end{equation}
Since $S$ is $E^*$-unitary, $e\leq u$ implies that $u$ is idempotent or $e=0$. If $e\neq 0$, the right hand side of (\ref{gamma-e}) is equal to
 \[
 \begin{cases}
  a_u \cdot b & \text{if }e\leq u\text{ and }u\in E,\\
  \hspace{2ex}0 & \text{ otherwise, }
 \end{cases}\\
\ = \
 \begin{cases}
  \lambda^\cE_{e,u}(a_u)b     & \text{ if } u\in E,\\
                 \hspace{3ex}      0 & \text{ otherwise,}
 \end{cases}
 \]
so we see that (\ref{gamma-e-u}) holds in this case. If $e=0$  (so $S$ has a $0$-element), then 
both sides of (\ref{gamma-e-u}) are equal to $0$ 
since $A_0=\{0\}$ by assumption.
\end{proof}
\begin{lemma}\label{lem:diagembedding}
 There is an embedding of $\CrKS{\cE}$ into $\CrKS{\cA}$ extending the inclusion $\csp{\cE}\subset\csp{\cA}$.
\end{lemma}
\begin{proof}
 Let $\sum_{f\in E}a_f\delta_f\in\csp{\cE}$. We need to prove that
 \begin{equation} \label{Phi-E}
 \big\|\Phi^\cE_\Lambda\Big(\sum_{f\in E}a_f\delta_f\Big)\big\|\, =\, \big\|\Phi^\cA_\Lambda\Big(\sum_{f\in E}a_f\delta_f\Big)\big\|.
 \end{equation}
 Since $\CrKS{\cE}=\CKS{\cE}$, cf.~Proposition \ref{prop:universalKS}, the expression on the left-hand side of (\ref{Phi-E}) is the same as the universal
 norm of $\sum_{f\in E}a_f\delta_f$ in  $\CKS{\cE}$. As  $\Phi^\cA_\Lambda$ restricts to a $*$-homomorphism of $C_c(\cE)$, we see that  the $\geq$ inequality in (\ref{Phi-E}) must hold.
 On the other hand, since $\gamma_e$ is an isometry for each $e\in E$, Lemma \ref{lem:conditional1} implies that
 \[
  \sup_{e\in E}\big\|\sum_{f\in E}\lambda^\cE_{e,f}(a_f)\big\|\,\leq\,\sup_{e\in E}\big\| \sum_{f\in E}\lambda^\cA_{e,f}(a_f)\big\|.
 \]
 This shows that $\leq$ inequality in (\ref{Phi-E}) holds.
\end{proof}

We will identify $\CrKS{\cE}$ with its canonical image in $\CrKS{\cA}$, and call it the \emph{diagonal} ($C^*$-subalgebra) of $\CrKS{\cA}$. 
Define $\fE:\csp{\cA}\to\csp{\cE}$ by
\begin{equation*}
 \fE\Big(\sum_{u\in S}a_u\delta_u\Big)=\sum_{e\in E}a_e\delta_e
\end{equation*}
for all $\sum_{u\in S}a_u\delta_u\in \csp{\cA}$. Moreover, define a positive linear map  $\fEKS$ from $\CrKS{\cA}$ into $\prod_{e\in E}\lin{A_e}$ by \[\fEKS\big((T_e)_{e\in E}\big)=(\gamma_e^*T_e\gamma_e)_{e\in E}\]
for all $(T_e)_{e\in E} \in \CrKS{\cA}$.

\begin{proposition}\label{condi-exp}
 The map $\fEKS:\CrKS{\cA}\to\prod_{e\in E}\lin{A_e}$ is faithful. 
 
 If $S$ is $E^*$-unitary and $A_0=\{0\}$ $($if $S$ has a $0$-element$)$, then $\fEKS$ satisfies
 \begin{equation}\label{eq:conditionalformula}
 \fEKS\Big(\Phi^{\cA}_\Lambda(g)\Big)=\Phi^\cE_\Lambda\Big(\fE(g)\Big)
 \end{equation}
 for all $g\in \csp{\cA}$. Moreover, in this case, $\fEKS$ is a faithful conditional expectation from $\CrKS{\cA}$ onto $\CrKS{\cE}$.
\end{proposition}
\begin{proof}
Let  $e\in E$, $T_e\in\lin{X_e}$ and $a\in A_e$. For each $s\in S_e$,  we have  \[(\gamma_s^*T_e\gamma_e)(a) = \big(T_e\gamma_e(a)\big)(s),\] so we get
\begin{align*}
 \Big\langle (\gamma_e^*T_e^*T_e\gamma_e)(a),a\Big\rangle_{A_e} &= \Big\langle T_e\gamma_e(a),T_e\gamma_e(a)\Big\rangle_{X_e}\\
 &=\sum_{s\in S_e} \big(T_e\gamma_e(a)\big)(s)^*\big(T_e\gamma_e(a)\big)(s)\\
 &=\sum_{s\in S_e} \big(\gamma_s^*T_e\gamma_e\big)(a)^*\big(\gamma_s^*T_e\gamma_e\big)(a).
\end{align*}
So we see that if $\gamma_e^*T_e^*T_e\gamma_e=0$, then $\gamma_s^*T_e\gamma_e=0$ for each $s\in S_e$. 

Consider $T=(T_e)_{e\in E}\in \CrKS{\cA}$. If $\gamma_s^*T_e\gamma_e=0$ for all $e\in E$ and $s\in S_e$, then for any $e\in E$ and $s,t\in S_e$, we have in particular that $\gamma_{st^*}^*T_{tt^*}\gamma_{tt^*}=0$, so Lemma~\ref{lem:factorize} and Lemma~\ref{lem:fourierfactorize} imply that $\gamma_s^*T_e\gamma_t=0$. Combining this with our first observation, we get that if $\gamma_e^*T_e^*T_e\gamma_e=0$ for each $e\in E$, then $\gamma_s^*T_e\gamma_t=0$ for each $e\in E$ and $s,t\in S_e$. 

Assume now that $\fEKS(T^*T) = 0$. This means that $\gamma_e^*T_e^*T_e\gamma_e=0$ for all $e\in E$. Hence, for  $e\in E$ and $\xi\in X_e$, we get
\[
 T_e\xi = \sum_{s\in S_e}  \gamma_s\gamma_s^*T_e\xi = \sum_{s\in S_e}\sum_{t\in S_e}\gamma_s\gamma_s^*T_e\gamma_t\gamma_t^*\xi=0.
\]
Thus $T_e=0$ for every $e\in E$, so $T=0$. This proves that $\fEKS$ is faithful.

Next, assume that $S$ is $E^*$-unitary and $A_0=\{0\}$ (if $S$ has a $0$-element). To show that \eqref{eq:conditionalformula} holds amounts to show that for any $e\in E$, we have
\[
\gamma_e^*\left(\sum_{u\in S}\lambda_{e,u}(a_u)\right)\gamma_e=\sum_{f\in E}\lambda_{e,f}(a_f).
\]
 for all $\sum_{u\in S}a_u\delta_u\in \csp{\cA}$. This follows readily from Lemma \ref{lem:conditional1}. It is then clear that the image of $\fEKS$ is $\CrKS{\cE}$. 
Note also  that $\fEKS$ is contractive since $\gamma_e$ is an isometry for each $e\in E$. 
Moreover, it is immediate from \eqref{eq:conditionalformula} that $\fEKS$ is a projection map. 
Hence, Tomiyama's theorem (see for instance \cite[Theorem II.6.10.2]{blackadar}) gives that  
$\fEKS$ is a conditional expectation.
 \end{proof}

\begin{remark}\label{groupbundle}
Suppose that $S$ is strongly $E^*$-unitary and and  $A_0=\{0\}$ (if $S$ has a zero). Let $\sigma$ be  an idempotent pure grading from $S^\times$ into a group $G$. Then for each $g\in G$ one can form the Banach space
\[
 B_g := \overline{\bigoplus_{s\in S,\,\sigma(s)=g}\Phi_\Lambda(A_s\delta_s)}\subset \CrKS{\cA}.
\]
It is straightforward to check that $\mathcal{B}:=\{B_g\}_{g\in G}$ is a Fell bundle over $G$, giving a $G$-grading for $\CrKS{\cA}$  in the sense of \cite[Definition 16.2]{exel14}. Moreover, since $\sigma$ is idempotent pure, we have $\{s\in S: \sigma(s) =1_G\}=E$, so $B_{1_G}=\CrKS{\cE}$. Since $\fEKS$ is faithful 
by the previous proposition, it then follows from \cite[Proposition 19.8]{exel14} that $\CrKS{\cA}$ is naturally isomorphic to the reduced cross sectional $C^*$-algebra $C^*_r(\mathcal{B})$ associated with $\mathcal{B}$.
\end{remark}

The following covariance property of $\fEKS$ will be useful later.

\begin{lemma}\label{covar}
Suppose $S$ is $E^*$-unitary and $A_0=\{0\}$ (if $S$ has a $0$ element). Then for all $s\in S$, $b\in A_s$ and $T\in\CrKS{\cA}$ we have
\begin{equation}\label{eq:covar}
\fEKS\big(\lambda_s(b)^*T\lambda_s(b)\big)=\lambda_s(b)^*\fEKS(T)\lambda_s(b).
\end{equation}
\end{lemma}
\begin{proof}
Let $s\in S$ and $b\in A_s$. 
Consider $\sum_{t\in S}a_t\delta_t\in \csp{\cA}$. Then for any $t\in S$ we have that $s^*ts=0$ if and only if $ss^*tss^*=0$. Moreover, $s^*ts\in E$ if and only if $ss^*tss^*\in E$; thus,
since $S$ is $E^*$-unitary and $t\geq ss^*tss^*$, we get that  $s^*ts\in E$ if and only if  $t\in E$ or  
 $s^*ts=0$.
Hence, using that $A_0=\{0\}$ (if $S$ has a $0$ element), we get
\begin{align*}
\fE\left((b\delta_s)^*\Big(\sum_{t\in S}a_t\delta_t\Big)(b\delta_s)\right)
&=\fE\left(\sum_
{t\in S,\, s^*ts\neq 0}(b^*\cdot a_t\cdot b)\delta_{s^*ts}\right)\\
&=\sum_
{t\in S, \, s^*ts\in E^\times}
(b^*\cdot a_t\cdot b)\delta_{s^*ts}\\
&=(b\delta_s)^*\fE\left(\sum_{t\in S}a_t\delta_t\right)(b\delta_s).
\end{align*}
Using equation (\ref{eq:conditionalformula}), we then see that equation (\ref{eq:covar}) holds whenever $T= \Phi_\Lambda^\cA(g)$ for some $g \in \csp{\cA}$.
By linearity and continuity of $\fEKS$ and density of $\Phi_\Lambda^\cA(\csp{\cA})$ in $\CrKS{\cA}$,
it then holds for all $T\in\CrKS{\cA}$.
\end{proof}

\section{Comparison with Exel's reduced cross sectional $C^*$-algebras}\label{sec:comparison}
{\it Throughout this section we consider a Fell bundle  $\cA=\{A_s\}_{s\in S}$ over an $E^*$-unitary inverse semigroup $S$ and assume that $A_0=\{0\}$ $($if $S$ has a $0$ element$)$.} Our aim is to show that Exel's $C^*_r(\cA)$ is a 
quotient of $\CrAlt{\cA}$ and that these $C^*$-algebras are canonically isomorphic under certain assumptions.
 
\smallskip As in the previous section, we let $\cE=\{A_e\}_{e\in E}$ denote the Fell bundle  obtained by restricting $\cA$ to the semilattice $E=E(S)$. From Proposition \ref{condi-exp}, we see that $\fEKS\big(\Phi^\cA_\Lambda(\cN_\cA)\big)=\Phi^\cE_\Lambda(\cN_\cE)$, and it easily follows that $\fEKS(\mathcal{I}_\cA)=\mathcal{I}_\cE$. 

\smallskip For any ideal $\mathcal{K}$ of $\CrKS{\cA}$ satisfying $\fEKS(\mathcal{K})=\mathcal{I}_\cE$ we can define a surjective linear map $\fE_\mathcal{K}:\CrKS{\cA}/\mathcal{K}\to\CrAlt{\cE}$ by \[\fE_\mathcal{K}(T+\mathcal{K})=\fEKS(T)+\mathcal{I}_\cE.\] It is straightforward to check that $\fE_\mathcal{K}$ is contractive. Note also that for each $T\in\Phi^\cA_\Lambda(\csp{\cE})$ we have
\[
\|T+\mathcal{I}_\cE\|=\|\fE_\mathcal{K}(T+\mathcal{K})\|\leq\|T+\mathcal{K}\|\leq\|T+\mathcal{I}_\cE\|
\]
where the last inequality uses that the map \[g\mapsto\Phi^\cA_\Lambda(g) +\, \mathcal{K}\] is a representation of $\csp{\cE}$ in $\CrKS{\cA}/\mathcal{K}$ and that $\CrAlt{\cE}=C^*(\cE)$. It follows that $\|T + \mathcal{I}_\cE\| = \|T+\mathcal{K}\|$ for each $T \in \CrKS{\cE}$, so
we can identify $\CrAlt{\cE}$ with
the image of $\CrKS{\cE}$ in the quotient $\CrKS{\cA}/\mathcal{K}$. Using Tomiyama's theorem (see for instance \cite[Theorem~II.6.10.2]{blackadar}), we get that $\fE_\mathcal{K}$ is a conditional expectation; in particular it is completely positive. 

\begin{proposition}
Define
\[
 \mathcal{J}_\cA=\{T\in\CrKS{\cA}:\fEKS(T^*T)\in \mathcal{I}_\cE\}.
\]
Then we have
\begin{equation}\label{tempideal}
\mathcal{J}_\cA=\{T\in\CrKS{\cA}\,:\, \fEKS(QTR)\in\mathcal{I}_\cE \, \text{ for all } Q,R\in \CrKS{\cA}\}\,.
\end{equation}
Thus $\mathcal{J}_\cA$ is an ideal of $\CrKS{\cA}$, satisfying $\mathcal{I}_\cA\subset\mathcal{J}_\cA$ and $\fEKS(\mathcal{J}_\cA)=\mathcal{I}_\cE$. 
Moreover, the conditional expectation $\fE_{\mathcal{J}_\cA}$ from
$\CrKS{\cA}/\mathcal{J}_\cA$ onto $\CrAlt{\cE}$ is faithful.
\end{proposition}
\begin{proof}
Let $\mathcal{K}$ be the ideal of $\CrKS{\cA}$ defined by the right hand side of equation \eqref{tempideal}. Then by using an approximate unit for $\CrKS{\cA}$ one easily deduce that $\fEKS(\mathcal{K})=\mathcal{I}_\cE$ and $\mathcal{K}\subset\mathcal{J}_\cA$.

Let $T\in\mathcal{J}_\cA$. 
Then we have
\begin{equation}\label{norm-zero}
\|\fE_\mathcal{K}(T^*T+\mathcal{K})\| = \|\fEKS(T^*T)+\mathcal{I}_\cE\| = 0
\end{equation}
since $\fEKS(T^*T)\in \mathcal{I}_\cE$. 
Consider now  $Q\in\CrKS{\cA}$. Then, by using 
the Cauchy-Schwarz inequality (cf.~\cite{lance95}) and equation (\ref{norm-zero}), we get
\begin{align*}
\|\fEKS\big((QT)^*(QT)\big)+\mathcal{I}_\cE\|^2&=
\|\fE_\mathcal{K}\big((QT)^*(QT)+\mathcal{K}\big)\|^2\\
&=\|\fE_\mathcal{K}\big((T+\mathcal{K})^*(Q^*QT+\mathcal{K})\big)\|^2 \\
&\leq \|\fE_\mathcal{K}(T^*T+\mathcal{K})\|\|\fE_\mathcal{K}((Q^*QT)^*(Q^*QT)+\mathcal{K})\|\\
&=0
\end{align*}
So $\fEKS((QT)^*(QT))\in \mathcal{I}_\cE$, hence $QT\in\mathcal{J}_\cA$.

Next, consider $R=\lambda_s(b)$ for $s\in S$ and $b\in A_s$. Lemma~\ref{covar} gives that $\fEKS(R^*T^*TR)=R^*\fEKS(T^*T)R$. So $\fEKS(R^*T^*TR)\in\mathcal{I}_\cA$ since $\fEKS(T^*T)\in \mathcal{I}_\cE\subset\mathcal{I}_\cA$ and $\mathcal{I}_\cA$ is an ideal. As the range of $\fEKS$ is $\CrKS{\cE}$ we also get that $\fEKS(R^*T^*TR)\in\CrKS{\cE}$, so $\fEKS(R^*T^*TR)\in\mathcal{I}_\cE$. By the Schwarz inequality (sometimes called the Kadison inequality), see for instance \cite[Proposition II.6.9.14]{blackadar}, we have
\[
\fEKS(TR)^*\fEKS(TR)\leq\fEKS((TR)^*TR)= \fEKS(R^*T^*TR)\in\mathcal{I}_\cE.
\]
Hence $\fEKS(TR)^*\fEKS(TR) \in \mathcal{I}_\cE$ since $\mathcal{I}_\cE$ (being an ideal) is a hereditary subalgebra of $\CrKS{\cE}$, and it therefore follows that
 $\fEKS(TR)\in\mathcal{I}_\cE$ (cf.~\cite[Proposition II.5.1.1]{blackadar}). 
By linearity and continuity of $\fEKS$ and density of $\Phi_\Lambda^\cA(\csp{\cA})$ in $\CrKS{\cA}$, we get that $\fEKS(TR)\in\mathcal{I}_\cE$ for all $R\in\CrKS{\cA}$.

If now $T\in\mathcal{J}_\cA$ and $Q, R\in\CrKS{\cA}$, then we get that $T':=QT\in\mathcal{J}_\cA$, and this implies that $\fEKS(QTR)=\fEKS(T'R)\in\mathcal{I}_\cE$. This shows that  $\mathcal{J}_\cA\subset\mathcal{K}$, hence that  $\mathcal{J}_\cA=\mathcal{K}$.

Since we now have shown that $ \mathcal{J}_\cA$  is an ideal of $\CrKS{\cA}$ satisfying $\fEKS(\mathcal{J}_\cA)= \mathcal{I}_\cE$, the canonical conditional expectation $\fE_{\mathcal{J}_\cA}$ from
$\CrKS{\cA}/\mathcal{J}_\cA$ onto $\CrAlt{\cE}$ is well defined. Showing that $\fE_{\mathcal{J}_\cA}$ is faithful amounts to verify that  $T^*T\in\mathcal{J}_\cA$ whenever $\fEKS(T^*T)\in\mathcal{I}_\cE$. This readily follows from the definition of $\mathcal{J}_\cA$ and the fact that $\mathcal{J}_\cA$ is an ideal of $\CrKS{\cA}$.
\end{proof}

\begin{proposition}\label{Psi-fEalt}
 Let $q_{\mathcal{J}_\cA} $ denote the quotient map from $\CrKS{\cA}$ onto $\CrKS{\cA}/\mathcal{J}_\cA$.
Then there exists a canonical isomorphism \[\Psi:\CrKS{\cA}/\mathcal{J}_\cA\to C^*_r(\cA)\] satisfying $\Psi\circ q_{\mathcal{J}_\cA}\circ \Phi_\Lambda^\cA = \iota^{\rm red}_\cA$.\end{proposition}
\begin{proof}
The strategy for proving the proposition is to show that there exists a $*$-homomorphism $\Psi:\CrKS{\cA}/\mathcal{J}_\cA\to C^*_r(\cA)$ and a linear map $\fEr:C^*_r(\cA)\to C^*(\cE)$ making the following diagram commute:
\begin{center}
\begin{tikzpicture}[>=angle 90, scale=2.2, text height=1.5ex, text depth=0.25ex]
\node (CrA)    at (0,3) {$C^*_r(\cA)$};
\node (CrJA) at (3,3) {$\CrKS{\cA}/\mathcal{J}_\cA$};
\node (CE)     at (0,0) {$C^*(\cE)$};
\node (CraltE) at (3,0) {$\CrAlt{\cE}$};
\node (CcA)    at (1,2) {$\csp{\cA}$};
\node (CrKSA)  at (2,2) {$\CrKS{\cA}$};
\node (CcE)    at (1,1) {$\csp{\cE}$};
\node (CrKSE)  at (2,1) {$\CrKS{\cE}$};
\path[right hook->,font=\scriptsize]
(CcA) edge node[above] {$\Phi^\cA_\Lambda$} (CrKSA)
(CcE) edge node[above] {$\Phi^\cE_\Lambda$} (CrKSE);
\path[->>,font=\scriptsize]
(CrKSA) edge node[above, sloped] {$q_{\mathcal{J}_\cA}$} (CrJA)
(CrKSE) edge (CraltE);
\path[->,font=\scriptsize]
(CcA) edge node[auto] {$\fE$} (CcE)
(CrKSA) edge node[auto] {$\fEKS$} (CrKSE)
(CrJA) edge node[auto] {$\fE_{\mathcal{J}_\cA}$} (CraltE)
(CraltE) edge node[above] {$=$} (CE)
(CcE) edge node[above, sloped] {$\iota_\cE$} (CE)
(CcA) edge node[above, sloped] {$\iota^{\rm red}_\cA$}(CrA);
\path[dashed,->,font=\scriptsize]
(CrA) edge node[auto] {$\fEr$} (CE)
(CrJA) edge node[above] {$\Psi$} (CrA);
\end{tikzpicture}
\end{center}
It will then follow that $\Psi$ is an isomorphism by considering the outer square in this diagram and using that $\fE_{\mathcal{J}_\cA}$ is faithful, as shown in the previous proposition. 

Let $\varphi$ be a pure state on $C^*(\cE)$. It is easy to deduce from equation (\ref{eq:stateextension}) that the functional  $\tilde{\varphi}$ on $\csp{\cA}$ defined in section \ref{sec:exelreduced} is given by \[\tilde{\varphi}=\varphi\circ \iota_{\cE}\circ\fE\]
 where $\iota_\cE$ denotes the canonical map from $\mathcal{C}_c(\cE)$ to $C^*(\cE)$. Moreover, it is straightforward to see that we have $\iota_{\cE}\circ\fE = \fE_{\mathcal{J}_\cA}\circ q_{\mathcal{J}_\cA}\circ\Phi_\Lambda^\cA$,
 so we get
\[
\tilde{\varphi}=\varphi\circ\fE_{\mathcal{J}_\cA}\circ q_{\mathcal{J}_\cA}\circ\Phi_\Lambda^\cA\,.
\]
Let $\varphi'=\varphi\circ\fE_{\mathcal{J}_\cA}\circ q_{\mathcal{J}_\cA}$. Then $\varphi'$ is a state on $\CrKS{\cA}$. As before, let $(\Upsilon_{\tilde{\varphi}}, H_{\tilde{\varphi}})$ be the GNS representation associated to $\tilde{\varphi}$, with $x\mapsto\hat{x}$ denoting the canonical map $\csp{\cA}\to H_{\tilde{\varphi}}$. Form also the GNS-representation $(\pi_{\varphi'}, H_{\varphi'})$ associated to $\varphi'$, with $T\mapsto \widehat{T}$ denoting the canonical map from $\CrKS{\cA}$ into $H_{\varphi'}$. For any $x\in\csp{\cA}$, we obtain
\[
 \|\hat{x}\|^2=\tilde{\varphi}(x^*x)=\varphi'\big(\Phi_\Lambda^\cA(x^*x)\big)=\|\widehat{\Phi_\Lambda^\cA(x)}\|^2.
\]
Since $\{\hat{x}:x\in\csp{\cA}\}$ is dense in $H_{\tilde{\varphi}}$, the assignment $\hat{x}\mapsto\widehat{\Phi_\Lambda^\cA(x)}$ extends to an isometry $V:H_{\tilde{\varphi}}\to H_{\varphi'}$. 

Consider now $g\in\csp{\cA}$. For any $x,y\in\csp{\cA}$ we get
\begin{align*}
 \Big\langle V^*\pi_{\varphi'}\big(\Phi_\Lambda^\cA(g)\big) V\hat{x},\hat{y}\Big\rangle&=\Big\langle \pi_{\varphi'}\big(\Phi_\Lambda^\cA(g)\big)\widehat{\Phi_\Lambda^\cA(x)},\widehat{\Phi_\Lambda^\cA(y)} \Big\rangle\\
 &=\varphi'\big(\Phi_\Lambda^\cA(y^*gx)\big)
 =\tilde{\varphi}(y^*gx)\\
 &=\Big\langle \Upsilon_{\tilde{\varphi}}(g)\hat{x},\hat{y}\Big\rangle.
\end{align*}
So $\Upsilon_{\tilde{\varphi}}(g)=V^*\pi_{\varphi'}\big(\Phi_\Lambda^\cA(g)\big)V$, and it follows that $\|\Upsilon_{\tilde{\varphi}}(g)\|\leq \|\pi_{\varphi'}\big(\Phi_\Lambda^\cA(g)\big)\|$ since $V$ is an isometry.  
Moreover, as $\varphi'$  annihilates $\mathcal{J}_\cA$, 
the kernel of $\pi_{\varphi'}$ contains $\mathcal{J}_\cA$, so we get $\|\pi_{\varphi'}\big(\Phi_\Lambda^\cA(g)\big)\|\leq\|q_{\mathcal{J}_\cA}(\Phi_\Lambda^\cA(g))\|$. Hence we conclude that  
\[\|\Upsilon_{\tilde{\varphi}}(g)\|\leq \|q_{\mathcal{J}_\cA}(\Phi_\Lambda^\cA(g))\|.\]
Since this holds for all $\varphi \in \mathcal{P}(C^*(\cE))$ we get $\|\iota^{\rm red}_\cA(g)\|\leq \|q_{\mathcal{J}_\cA}(\Phi_\Lambda^\cA(g))\|$. It follows that there exists a $*$-homomorphism $\Psi:\CrKS{\cA}/\mathcal{J}_\cA\to C^*_r(\cA)$ satisfying $\Psi\big(q_{\mathcal{J}_\cA}(\Phi_\Lambda^\cA(g))\big) = \iota^{\rm red}_\cA(g)$ for all $g \in\csp{\cA}$, as desired.
 
Next, we will show that the map $\fEr':\iota_\cA^{\rm red}(\csp{\cA})\to C^*(\cE)$ given by $\fEr'(\iota_\cA^{\rm red}(g))=\iota_\cE(\fE(g))$ is well defined, linear and contractive. By density, it will then extend to a (contractive) linear map $\fEr:C^*_r(\cA)\to C^*(\cE)$, as desired.

To see that $\fEr'$ is well defined, note that if $g\in\csp{\cA}$ and $\iota_\cA^{\rm red}(g)=0$, then $0=\tilde{\varphi}(x^*gy)=\varphi(\iota_\cE(\fE(x^*gy)))$ for all $x,y\in\csp{\cA}$ and all $\varphi \in \mathcal{P}(C^*(\cE))$. Letting $x$ and $y$ range over $\csp{\cE}$ so that $\fE(x^*gy)=x^*\fE(g)y$ (which follows from Proposition \ref{condi-exp} since $\fEKS$ is a conditional expectation), and using the density of $\iota_\cE(\csp{\cE})$ in $C^*(\cE)$ we get $\iota_\cE(\fE(g))=0$. It readily follows that $\fEr'$ is well defined, and its linearity is then obvious.   

Further, consider $g\in\csp{\cA}$. For $x,y\in\csp{\cE}$ with $\|\iota_\cE(x)\|,\|\iota_\cE(y)\|\leq 1$, we have
 \begin{align*}
  \|\iota_\cE(x^*\fE(g)y)\|&=\|\iota_\cE(\fE(x^*gy))\|=\sup_{\varphi\in\mathcal{P}(C^*(\cE))}\varphi(\iota_\cE(\fE(x^*gy)))\\
	&\leq \sup_{\varphi\in\mathcal{P}(C^*(\cE))}\|\Upsilon_{\tilde{\varphi}}(g)\|=\|\iota^{\rm red}_\cA(g)\|\,.
 \end{align*}
For every $\varepsilon>0$ it is not difficult to see that we can find  $x$ and $y$ as above such that \[\|\iota_\cE(\fE(g))\|\leq\|\iota_\cE(x^*\fE(g)y)\|+\varepsilon,\] so  we get 
\[\|\fEr'(\iota_\cA^{\rm red}(g))\| =\|\iota_\cE(\fE(g))\|\, \leq\,\|\iota_\cE(x^*\fE(g)y)\|+\varepsilon \,\leq\, \|\iota^{\rm red}_\cA(g)\| + \varepsilon\,.\]
Thus we conclude that $\fEr'$ is contractive.

The reader will have no problem to check that the maps $\Psi$ and $\fEr$ we have constructed make the above diagram commutative, thus finishing the proof.
\end{proof}

From Proposition \ref{Psi-fEalt} and its proof we get the following result which may be worthy of being stated separately. It is proved for saturated Fell bundles over unital inverse semigroups in \cite{buss_exel_meyer15} using a different approach.

\begin{proposition}\label{faith-cond}
There exists a faithful conditional expectation \[\fEr:C^*_r(\cA)\to C^*(\cE)\]
satisfying $\fEr\circ\iota_\cA^{\rm red} = \iota_\cA\circ\fE$.
\end{proposition}

Since $\mathcal{I}_\cA\subset\mathcal{J}_\cA$, there is a natural surjective $*$-homomorphism $p_\cA$ from $\CrAlt{\cA}= \CrKS{\cA}/\mathcal{I}_\cA$ onto $\CrKS{\cA}/\mathcal{J}_\cA$. Using Propositions \ref{Psi-fEalt} and \ref{faith-cond},  the relationship between the reduced $C^*$-algebra $\CrAlt{\cA}$
introduced in the present article and Exel's $C^*_r(\cA)$ can be described as follows:

\begin{theorem} \label{Alt-Exel}
There exists a surjective canonical $*$-homomorphism \[\Psi':\CrAlt{\cA}\to C^*_r(\cA)\] satisfying $\Psi'\circ \Psi_{\Lambda^{\rm alt}} = \Psi_{\Pi^{\rm red}}$. 

Moreover, the  conditional expectation $\fEr^{\rm alt}:\CrAlt{\cA}\to\CrAlt{\cE}=C^*(\cE)$ given by $\fEr^{\rm alt} = \fEr\circ \Psi'$ is canonical in the sense that \[\fEr^{\rm alt}\big( \Phi^\cA_{\Lambda^{\rm alt}}(g)\big)
= \Phi^\cE_{\Lambda^{\rm alt}}(\fE(g))\] for all $g \in C_c(\cA)$, and the following conditions are equivalent:
\begin{itemize}
\item[(i)] $\Psi'$ is an isomorphism;

\smallskip \item[(ii)] $\mathcal{I}_\cA=\mathcal{J}_\cA$;

\smallskip \item[(iii)] $\fEr^{\rm alt}$ is faithful.
 \end{itemize}
\end{theorem}
\begin{proof} It suffices to set $\Psi' = \Psi \circ p_\cA$ and observe that $\fEr^{\rm alt} = \fE_{\mathcal{J}_\cA}\circ p_\cA$.
\end{proof}

We don't know whether $\CrAlt{\cA}$ is isomorphic to $C^*_r(\cA)$ in general. When $S$ is strongly $E^*$-unitary this happens quite often.

\begin{corollary}
Assume $S$ is strongly $E^*$-unitary and let $\sigma:S^\times\to G$ be an idempotent pure grading into a group $G$. Let $\mathcal{B}$ be the associated Fell bundle over $G$ defined in Remark~\ref{groupbundle}. If  $G$ is exact \cite{brown_ozawa08}, or if $\mathcal{B}$ satisfies Exel's approximation property \cite{exel97},
then  $\CrAlt{\cA}$ is canonically isomorphic to $ C^*_r(\cA)$.
\end{corollary}
\begin{proof}
By using  \cite[Theorem 23.7]{exel14} if $G$ is exact, or  \cite[Proposition 23.6]{exel14} if $\mathcal{B}$ satisfies Exel's approximation property, one deduces easily that $\mathcal{I}_\cA=\mathcal{J}_\cA$ after making appropriate identifications of these ideals in $C^*_r(\mathcal{B})$. Hence, the result follows from Theorem \ref{Alt-Exel}.
\end{proof}

\bibliographystyle{plain}
\bibliography{bibliography}

\begin{thebibliography}{10}

\bibitem{blackadar}
B.~Blackadar.
\newblock {\em Operator algebras}, volume 122 of {\em Encyclopaedia of
  Mathematical Sciences}.
\newblock Springer-Verlag, Berlin, 2006.
\newblock Theory of $C{^{*}}$-algebras and von Neumann algebras, Operator
  Algebras and Non-commutative Geometry, III.

\bibitem{brown_ozawa08}
Nathanial~P. Brown and Narutaka Ozawa.
\newblock {\em {$C^*$}-algebras and finite-dimensional approximations},
  volume~88 of {\em Graduate Studies in Mathematics}.
\newblock American Mathematical Society, Providence, RI, 2008.

\bibitem{fleming_fountain_gould99}
Sydney Bulman-Fleming, John Fountain, and Victoria Gould.
\newblock Inverse semigroups with zero: covers and their structure.
\newblock {\em J. Austral. Math. Soc. Ser. A}, 67(1):15--30, 1999.

\bibitem{buss_exel11}
Alcides Buss and Ruy Exel.
\newblock Twisted actions and regular {F}ell bundles over inverse semigroups.
\newblock {\em Proc. Lond. Math. Soc. (3)}, 103(2):235--270, 2011.

\bibitem{buss_exel_12a}
Alcides Buss and Ruy Exel.
\newblock Fell bundles over inverse semigroups and twisted \'etale groupoids.
\newblock {\em J. Operator Theory}, 67(1):153--205, 2012.

\bibitem{buss_exel12}
Alcides Buss and Ruy Exel.
\newblock Inverse semigroup expansions and their actions on {$C^*$}-algebras.
\newblock {\em Illinois J. Math.}, 56(4):1185--1212, 2012.

\bibitem{buss_exel_meyer15}
Alcides Buss, Ruy Exel, and Ralf Meyer.
\newblock Reduced ${C}^*$-algebras of {F}ell bundles over inverse semigroups.
\newblock 2015.
\newblock Preprint arXiv:1512.05570v1.

\bibitem{buss_meyer14}
Alcides Buss and Ralf Meyer.
\newblock Inverse semigroups actions on groupoids.
\newblock 2014.
\newblock Preprint arXiv:1410.2051, to appear in Rocky Mountain J. Math.

\bibitem{exel97}
Ruy Exel.
\newblock Amenability for {F}ell bundles.
\newblock {\em J. Reine Angew. Math.}, 492:41--73, 1997.

\bibitem{exel11}
Ruy Exel.
\newblock Noncommutative {C}artan subalgebras of {$C^*$}-algebras.
\newblock {\em New York J. Math.}, 17:331--382, 2011.

\bibitem{exel14}
Ruy Exel.
\newblock {\em Partial Dynamical Systems, Fell Bundles and Applications}.
\newblock 350pp, arXiv:1511.04565v1. Also available online from www.
  mtm.ufsc.br/~exel/publications/., 2014.

\bibitem{hewitt_ross70}
Edwin Hewitt and Kenneth~A. Ross.
\newblock {\em Abstract harmonic analysis. {V}ol. {II}: {S}tructure and
  analysis for compact groups. {A}nalysis on locally compact {A}belian groups}.
\newblock Die Grundlehren der mathematischen Wissenschaften, Band 152.
  Springer-Verlag, New York-Berlin, 1970.

\bibitem{khoshkam_skandalis04}
Mahmood Khoshkam and Georges Skandalis.
\newblock Crossed products of {$C^*$}-algebras by groupoids and inverse
  semigroups.
\newblock {\em J. Operator Theory}, 51(2):255--279, 2004.

\bibitem{lance95}
E.~C. Lance.
\newblock {\em Hilbert {$C^*$}-modules}, volume 210 of {\em London Mathematical
  Society Lecture Note Series}.
\newblock Cambridge University Press, Cambridge, 1995.
\newblock A toolkit for operator algebraists.

\bibitem{lawson}
Mark~V. Lawson.
\newblock {\em Inverse semigroups}.
\newblock World Scientific Publishing Co., Inc., River Edge, NJ, 1998.
\newblock The theory of partial symmetries.

\bibitem{paterson}
Alan L.~T. Paterson.
\newblock {\em Groupoids, inverse semigroups, and their operator algebras},
  volume 170 of {\em Progress in Mathematics}.
\newblock Birkh\"auser Boston, Inc., Boston, MA, 1999.

\bibitem{re80}
Jean Renault.
\newblock {\em A groupoid approach to {$C^{\ast} $}-algebras}, volume 793 of
  {\em Lecture Notes in Mathematics}.
\newblock Springer, Berlin, 1980.

\bibitem{sieben}
N{\'a}ndor Sieben.
\newblock {$C^\ast$}-crossed products by partial actions and actions of inverse
  semigroups.
\newblock {\em J. Austral. Math. Soc. Ser. A}, 63(1):32--46, 1997.

\end{thebibliography}

\end{document}